\documentclass[twoside,11pt]{amsart}
\usepackage{amsmath,latexsym,amssymb,times,enumerate,hyperref}

\usepackage{verbatim}

\def\sqr#1#2{{\vcenter{\hrule height.#2pt
        \hbox{\vrule width.#2pt height#1pt \kern#1pt
                \vrule width.#2pt}
        \hrule height.#2pt}}}

\usepackage{tikz}
\usepackage{pst-plot}
\usepackage{pst-node,pst-text,pst-3d,pstricks}
\usetikzlibrary{patterns}

\def\sqr#1#2{{\vcenter{\hrule height.#2pt
        \hbox{\vrule width.#2pt height#1pt \kern#1pt
                \vrule width.#2pt}
        \hrule height.#2pt}}}
\def\square{\mathchoice\sqr64\sqr64\sqr{4}3\sqr{3}3}
\def\QED{\hfill$\square$}

\def\n{{\mathfrak n}}

\newtheorem{Theorem}{Theorem}[section]

\newtheorem{Lemma}[Theorem]{Lemma}
\newtheorem{Corollary}[Theorem]{Corollary}
\newtheorem{Proposition}[Theorem]{Proposition}

\newtheorem{Notation}[Theorem]{Notation}
\newtheorem{Assumptions and Discussion}[Theorem]{Assumptions and Discussion}

\newtheorem{Remark}[Theorem]{Remark}
\newtheorem{Example}[Theorem]{Example}
\newtheorem{Definition}[Theorem]{Definition}

\newtheorem{Algorithm}[Theorem]{Algorithm}

\newcommand{\pd}{\mathop{\mathrm{pd}}\nolimits}
\renewcommand{\H}{\mathcal{H}}
\newcommand{\Q}{\mathcal{Q}}
\newcommand{\OO}{\mathcal{O}}
\newcommand{\s}{\mathcal{S}}

\begin{document}

\baselineskip=16pt

\title[Projective Dimension of Strings and Cycles]
{\Large\bf Projective Dimension of String and Cycle Hypergraphs}

\thanks{AMS 2010 {\em Mathematics Subject Classification}:
Primary 13A30; Secondary 13H15, 13B22, 13C14, 13C15, 13C40.}
\thanks{$^1$ P. Mantero gratefully acknowledges the support of an AMS-Simons Travel Grant}

\author{Kuei-Nuan Lin}
\address{Smith College \\ Department of Mathematics and Statistics \\ Northampton, MA 01063, USA}
\email{klin@smith.edu\newline
\indent{\it URL:} \href{http://www.math.smith.edu/~klin/}{\tt http://math.smith.edu/$\sim$klin/}
}


\author[Paolo Mantero]{Paolo Mantero$^{1}$}
\address{University of California, Riverside \\ Department of Mathematics \\
 Riverside, CA 92521}
\email{mantero@math.ucr.edu\newline
\indent{\it URL:} \href{http://math.ucr.edu/~mantero/}{\tt http://math.ucr.edu/$\sim$mantero/}}

\date{October 7, 2013}

\vspace{-0.1in}

\begin{abstract}
We present a closed formula and a simple algorithmic procedure to compute the projective dimension of square-free monomial ideals associated to string or cycle hypergraphs. As an application, among these ideals we characterize all the  Cohen-Macaulay ones.
\end{abstract}

\maketitle

\vspace{-0.2in}

\section{Introduction}

Let $R=k[\mathcal{A}]$ be a polynomial ring over a field $k$ with
indeterminate set $\mathcal{A}$.
In the present paper, we are interested in computing the projective dimension $\pd(R/I)$ of a monomial ideal $I$. Since it is well-known that a polarization $P(I)$ has the same Betti numbers as $I$, we can restrict our attention to square-free monomial ideals $I$. To determine $\pd(R/I)$, we study the hypergraph associated to $I$ via a construction introduced by Kimura et. al \cite{KTY}. The construction is defined as follows: each minimal generator of $I$ corresponds to a vertex of the hypergraph, whereas the faces are defined by the relationships between the minimal generators of $I$ (see Definition \ref{H}). The hypergraph $\H$ obtained in this way is the dual hypergraph (cf. \cite{Be}) of the hypergraph  whose edge ideal is $I$.

This association of hypergraphs has been employed to prove the equality of arithmetical rank and projective dimension of classes of square-free monomial ideals \cite{KTY}, \cite{KTY2}, \cite{KRT}, study the normality of toric rings of square-free monomial ideals \cite{HL}, combinatorially classify almost complete intersection square-free monomial ideals, Cohen-Macaulay square-free monomial ideals of deviation at most two, and square-free monomial for which $\pd(R/I)\geq \mu(I)-1$ \cite{KTY}, \cite{KTY2}, \cite{KTY3}, compute the regularity of square-free monomial ideals \cite{LM}, and classify the square-free monomial ideals generated in degree two or having deviation at most two that are licci \cite{KTY3}.

Another reason for adopting this construction is the following: it is well-know that the minimal free resolution of the ideal $I$ is built from the syzygy relations of the minimal generators of $I$, see, for instance, \cite{GPW}. The hypergraph associated to $I$ provides a clear pictorial view of the relations between the minimal generators of $I$, hence it seems well-suited to compute projective dimensions. Moreover, since two ideals associated to the same hypergraph have the same projective dimension (Corollary \ref{same}), one can study the different classes of hypergraphs to understand the projective dimension of square-free monomial ideals. 

In the present paper, we focus on determining the projective dimension of ideals $I$ whose associated hypergraph $\mathcal{H}$ is a string or a cycle (see Definitions \ref{str} and \ref{Defcycle}). 

In Section 2 we prove several lemmas determining a procedure to compute recursively the projective dimension of any ideal associated to a string hypergraph (see Remark \ref{proc}). 

We devote Section 3 to prove our first main result, Theorem \ref{string}, providing the following closed formula for the projective dimension $\pd(\H)$ of any monomial ideal associated to a string $\H$:
$$\pd(\H)=\mu(\H)-b(\H)+M(\H).$$
Here, $\mu(\H)$ is the number of vertices of $\H$, $b(\H)$ depends on the number and distribution of open vertices in $\H$, and $M(\H)$ is an invariant dubbed the {\it modularity} of $\mathcal{H}$ -- see Sections 2 and 3, and Definition \ref{modularity}.
All these numbers can be easily computed. Hence, the above formula provides a fast way to compute projective dimensions via simple combinatorial counting (see, for instance, Examples \ref{DS2}, \ref{mod}, and Corollary \ref{examples}).

In Section 4 we turn our attention to the projective dimension of ideals associated to cycles. Remarkably, the formula we obtain in this case is the same formula found for strings (Theorem \ref{cycles2}) although the proof of this second main result is more involved. We then classify all Cohen-Macaulay ideals associated to string or cycle hypergraphs (Theorem \ref{CM}).
Note that our main results are not covered by previous work of Dao-Schweig \cite{DS} (see Examples \ref{DS2} and \ref{DS}) or Kimura-Terai-Yoshida \cite{KTY}, \cite{KTY2}, which highlight different focuses or approaches.

In Section 5 we write the algorithmic procedures to compute the projective dimension of ideals associated to strings and cycles and provide several examples. The reason for writing these algorithmic processes 
is that the computation of the modularity of $\H$ is time-consuming, whereas Algorithms \ref{alg} and \ref{alg2} are much faster procedures.

As a final remark, the main results of this paper yield that $\pd(R/I)$ is independent of the characteristic of the base field $k$ for all square-free monomial ideals $I$ associated to strings or cycles.

We would like to thank K. Kimura for relevant suggestions that lead to improvements of the paper, H. T. H\`a for helpful observations, and J. McCullough for discussions regarding early material of this manuscript.

\section{First results}

The notion of hypergraph employed in this paper was introduced by Kimura, Terai and Yoshida \cite{KTY}.  Let $[\mu]$ denote the subset $\{1,\ldots,\mu\}$ of $\mathbb N$.
\begin{Definition}\label{H}
Let $V = [\mu]$ and $\mathcal P(V)$ denote its power set.  A subset $\H\subseteq \mathcal{P}(V)$ is called a hypergraph with vertex set $V$ if	
$\bigcup\limits_{F\in \H}F=V$.

$\H$ is {\em separated} if, in addition, for every $1\leq j_1<j_2\leq \mu$, there exist faces $F_1,F_2\in \H$ so that $j_1\in F_1 \cap (V\setminus F_2)$ and $j_2\in F_2 \cap (V\setminus F_1)$.
\end{Definition}

Let $I\subseteq k[\mathcal{A}]$ be a square-free monomial ideal, where $\mathcal{A}$ is an alphabet set. Following \cite{KTY}, we associate to $I$ a (unique) separated hypergraph in the following way. Let $\{m_1,\dots,m_{\mu}\}$ be a minimal monomial generating set of $I$, set $V =[\mu]$, one defines the hypergraph $\H(I)$ as follows:
$F$ is a face of $\H(I)$ if and only if there is an $a\in\mathcal{A}$ such that for all $j\in V$, $m_j$ is divisible by $a$ if and only if $j\in F$. The hypergraph $\H(I)$ is separated because of the minimality of the generating set $m_1,\dots,m_{\mu}$. It is worth noting that $\H(I)$ is the dual hypergaph (see \cite{Be}) of the hypergraph whose edge ideal is $I$.

Conversely, to any separated hypergraph $\H$ one can associate a square-free monomial ideal which, by separability of $\H$, is minimally generated by the vertices of the hypergraph. The ideal obtained in this way is far from being unique. For instance, the two ideals $I=(ab,bc)$ and $J=(abcde,def)$ correspond to the same separated hypergraph $\H$ on the vertex set $V=\{1,2\}$ whose faces are $\{1\},\{2\},\{1,2\}$.

In \cite {LM}, labelled hypergraphs were introduced to make the correspondence between square-free monomial ideals and labelled separated hypergraphs become one to one. The authors investigated the regularity of square-free monomial ideals, for which the labelling is, in fact, essential. In this paper, however, we are interested in the projective dimension of square-free monomial ideals, and Corollary \ref{same} shows that, for this purpose, the labelling is not needed. In fact, more generally, we show that all the total Betti numbers are independent of the labelling.
Recall that $\beta_l(R/J)={\rm dim}_k{\rm Tor}_l^R(R/J,k)$ denote the $l$-th total Betti number of $R/J$. 
\begin{Proposition}\label{same2}
If $I_1$ and $I_2$ are square-free monomial ideals associated to the same separated hypergraph $\H$, then $\beta_l(I_1)=\beta_l(I_2)$ for every $l$.
\end{Proposition}
\begin{proof}
After possibly enlarging the polynomial ring, we may assume $I_1$ and $I_2$ lie in the same polynomial ring $R$. It suffices to show $\beta_l(R/I_1)=\beta_l(R/I_2)$ for every $l$. 
\medskip

{\bf Claim.} {\it If $I=(m_1,\ldots,m_u)\subseteq R$ is a square-free monomial ideal, then there exists a square-free monomial ideal $I'=(m_1',\ldots,m_{u}') \subseteq R$ such that $I$ and $I'$ have the same associated hypergraph, for each face $F_i$, there exists a unique variable $a_i$ such that $a_i|m_j'$ if and only if $j\in F_i$, and $\beta_l(R/I)=\beta_l(R/I')$ for every $l$.}
\medskip

We prove the claim. For each face $F_i$ of the hypergraph associated to $I$, there exists a maximal monomial $g_i$ such that $g_i|m_j$ if and only if $j\in F_i$. For every $i$, let $a_i$ be a variable dividing $g_i$. The natural inclusion $S=k[\{g_i\}]\subseteq R$ makes $R$ a free module over $S$, because the $g_i$ are square-free and for every variable $b$ of $R$, either $b$ does not divide any of the $g_i$ or $b$ appears in the monomial support of exactly one of the $g_i$. We then  have $\beta_l(S/I)=\beta_l(R/I)$ for every $l$. 

Next, since the $g_i$ form a regular sequence of ${\rm dim}\,R'$ elements, the homomorphism of $k$-algebras $\Phi:S\rightarrow R'=k[\{a_i\}]$ given by $\Phi(g_i)=a_i$ is actually an isomorphism. Setting $\Phi(I)R'=I'=(m_1',\ldots,m_{u}')$, we then have $\beta_l(S/I)=\beta_l(R'/I')$ for every $l$. Combining these equalities with the above, we obtain $\beta_l(R/I)=\beta_l(S/I)=\beta_l(R'/I')$. Moreover, since $R'\subseteq R$ is faithfully flat, we obtain $\beta_l(R'/I')=\beta_l(R/I')$ for every $l$, which yields $\beta_l(R/I)=\beta_l(R/I')$. Finally, by construction, for each face $F_i$, there exists a unique variable $a_i$ such that $a_i|m_j'$ if and only if $j\in F_i$.  This concludes the proof of the claim.

Now, by the Claim, we can assume $I_1=(m_{1},\cdots,m_{\mu})\subseteq R$ and $I_{2}=(n_{1},\cdots,n_{\mu})\subseteq R$ have the following property: 
for each face $F_i$ of $\H$, there exist unique variables $a_i$ and $b_i$ of $R$ such that $a_i|m_j$ if and only if $j\in F_i$ and $b_i|n_j$ if and only if $j\in F_i$.
Let $R_1= k\{[a_i\}]$ and $R_2=k[\{b_i\}]$, and note that for every $j=1,2$, we have $I_j\subseteq R_j$, $\beta_l(R/I_j)=\beta_l(R_j/I_j)$ for every $l$, and every variable of $R_j$ divides at least one minimal generator of $I_j$. It then suffices to show $\beta_l(R_1/I_1)=\beta_l(R_2/I_2)$. This holds because the homomorphism of $k$-algebras $\Phi:R_1\rightarrow R_2$ defined by $\Phi(a_i)=b_i$ for every $i$ is actually an isomorphism and, by definition of $\Phi$, one has $\Phi(I_1)R_2=I_2$.
\end{proof}

\begin{Remark}
The Claim in the proof of Proposition \ref{same2} could be rephrased using the terminology of \cite{KRT}. Indeed, it proves the existence of a minimal generic set $G$ (see \cite[Definition~5.1]{KRT}) with respect to the property $\mathcal B$ of having fixed total Betti numbers $\{\beta_i\}_{i\in\mathbb N}$, and satisfying the additional property that the ideals of $G$ are generated in minimal degrees.
\end{Remark}

\begin{Corollary}\label{same}
If $I_1$ and $I_2$ are square-free monomial ideals associated to the same separated hypergraph $\H$, then  $\pd(R/I_1)=\pd(R/I_2)$.
\end{Corollary}

We now set a few pieces of notation that will be employed throughout this paper. Since non-separated hypergraphs correspond to non-minimal generating sets of square-free monomial ideals,  we may harmlessly assume all hypergraphs are separated, and then by 'hypergraph' we always mean 'separated hypergraph'.
\begin{Notation}\label{not}
Let $\H$ be a hypergraph, and let $1,2,\ldots,\mu$ be the vertices of $\H$.
 \begin{itemize}
\item We set $\pd(\H)$ for $\pd(R/I)$, where $I$ is any square-free monomial ideal associated to $\H$. By Corollary \ref{same} this number is well-defined. We call this number the {\em projective dimension} of $\H$.
\item $I(\H)$ is an ideal minimally generated by $m_1 ,...m_{\mu}$ such that for every face $F$ in $\H$, there is a unique variable $a$ such that $a|m_k$ if and only if $k$ is in $F$. To compute $\pd(\H)$ we will always compute $\pd(R/I(\H))$.
\item $m_1,\ldots,m_{\mu}$ are the minimal generators of $I(\H)$ corresponding to the vertices $1,2,\ldots,\mu$;
\item We define $I_i=(m_{i+1},\ldots,m_{\mu})$, and we let $\H_i$ be the hypergraph associated to $I_i$.
\item We set $J_1=I_1:m_1$ and $\Q_1$ is the hypergraph associated to $J_1$. For the reader's convenience we will sometimes write $\Q_1=\H_1:v_1$.
\end{itemize}
\end{Notation} 

We will frequently need to distinguish open vertices from closed vertices of a hypergraph, hence we recall these two definitions. A vertex $v$ of $\H$ is {\it closed} if $\{v\}$ is a face of $\H$, that is, if $\{v\}\in \H$. A vertex $v$ is {\it open} if it is not closed. In the figures, closed vertices are represented by filled dots, and open vertices by unfilled dots. We also recall that a vertex $v_1$ is a {\it neighbor} of a vertex $v_2$ if there is a face of $\H$ containing both $v_1$ and $v_2$. We begin by stating a lemma that the first author learnt from the anonymous referee of the paper \cite{LM}.
\begin{Lemma}\label{ref}
Let $\H$ be a hypergraph. If $\{1\}\in \H$, then $\pd(\H)={\rm max}\{\pd(\H_1),\pd(\Q_1)+1\}$.
\end{Lemma}
\begin{proof}
Recall that $R=k[\mathcal A]$. By assumption, there exists a variable $a \in \mathcal A$ such that $a| m_1$ and $a\nmid m_i$ for any $i>1$.
We now consider $R=k[\mathcal A]$ multigraded via the following grading: ${\rm deg}(a)=(0,1)$ and  ${\rm deg}(b)=(1,0)$ for every $b\in \mathcal A$ with $b\neq a$. Let $d_1+1\geq 1$ be the total degree of $m_1$.  The statement follows if we show that the mapping cone of the following short exact sequence of modules (with multigraded degree $0$ maps) gives a minimal free resolution of $R/I(\H)$
$$0\longrightarrow R/J_1[-(d_1,1)] \stackrel{\cdot m_1}{\longrightarrow} R/I_1 \longrightarrow R/I(\H) \longrightarrow 0$$
To see this, let $F_{\bullet}$ and $G_{\bullet}$ be minimal graded free resolutions of $R/J_1$ and $R/I_1$, respectively. Let $\alpha:G_{\bullet}[-(d_1,1)]\rightarrow F_{\bullet}$ be a lifting of the multigraded homogeneous injection $R/J_1[-(d_1,1)]\stackrel{\cdot m_1}{\rightarrow}R/I_1$. 
Then every graded twist appearing in $G_{\bullet}[-(d_1,1)]$ has bi-grading $(*,1)$, whereas all twists in $F_{\bullet}$ have bi-grading $(*,0)$. Hence, for every $i$, every non-zero entry of a matrix corresponding to the map $G_i[-(d_1,1)] \rightarrow F_i$ is divisible by $a$. Hence, $\alpha$ is minimal, and then the mapping cone gives a minimal free resolution of $R/I(\H)$.
\end{proof}

We recall that a hypergraph $\H$ is {\em saturated} if $\{i\}\in \H$ for all $i=1,\ldots,\mu$, that is, if all the vertices of $\H$ are closed. The next result is well-known, see, for instance, \cite{F}.
\begin{Proposition}\label{saturated}
Let $\H$ be an hypergraph with $\mu$ vertices, then $\pd(\H)\leq \mu$ and equality holds if and only if $\H$ is saturated.
\end{Proposition}

Given two hypergraphs, $\H$ and $\H'$, one writes $\H'\subseteq \H$ if all the faces of $\H'$ are also faces of $\H$.
The following statement is well-known and similar to the statement of \cite[Proposition~3.2]{KRT}). We include a short proof for the reader's convenience.
\begin{Lemma}\label{loc}
If $\H'\subseteq\H$ are hypergraphs with $\mu(\H)=\mu(\H')$, then $\pd(\H')\leq\pd(\H)$.
\end{Lemma}
\begin{proof}
Let $\H'=\{F_1,\ldots,F_s\}$ and $\H=\H'\cup \{G_1,\ldots,G_r\}$, where $G_j\notin \H'$ for any $j$.
Let $\mathcal B=\{a_1,\ldots,a_s,b_1,\ldots,b_r\}$ be a set of variables over $k$, and let $R=k[\mathcal B]$ and $R'=k[\mathcal A]$ where $\mathcal A=\{a_1,\ldots,a_s\}$.
For each $i=1,\ldots,\mu=\mu(\H)$, let $C_i=\{j\in \{1,\ldots,s\}\,|\, i\in G_j \}$ and $D_i=\{j\in \{1,\ldots,r\}\,|\, i\in G_j \}$, set $m_i=\prod_{j\in C_i}a_j \cdot \prod_{j\in D_i}b_j$ and $m_i'=\prod_{j\in C_i}a_j$. By construction, the ideal $I=(m_1,\ldots,m_{\mu})$ is associated to the hypergraph $\H$, and $I'=(m_1',\ldots,m_{\mu}')$ is associated to $\H'$, hence, by Corollary \ref{same} we have $\pd(\H')=\pd(R'/I')=\pd(R/I')$ and $\pd(\H)=\pd(R/I)$. Let $M'$ be the maximal homogeneous ideal of $R'=k[\mathcal A]$. The natural inclusion $R'\subseteq R$ induces the faithfully flat extension $R'_{M'}\subseteq R_{M'R}$ and, clearly, $I'R_{M'R}=IR_{M'R}$, then we have
$$\pd(\H')=\pd(R'/I')=\pd(R'_{M'}/I'_{M'})=\pd(R_{M'R}/IR_{M'R})\leq \pd(R/I)=\pd(\H),$$
where the second equality follows because $M'$ is the unique homogeneous maximal ideal of $R'$.
\end{proof}

We record the following special case of Lemma \ref{loc} for future use.
\begin{Corollary}\label{compare}
Let $\H$ be any hypergraph and assume $\{i\}\notin \H$ for some $i$. If one sets $\overline{\H}=\H\cup \{i\}$, then 
$\pd(\H)\leq \pd(\overline{\H})$.
\end{Corollary}

 To prove the next statement, we need to recall Taylor's resolution. The Taylor's resolution $\mathbb T.$ of a square-free monomial ideal $I$ minimally generated by monomials $m_1,\ldots,m_{\mu}$ is defined as follows. $T_1$ is a free $R$-module of rank $\mu$ with basis $e_1,\ldots,e_{\mu}$, for all $i$ we set $T_i = \bigwedge^iT_1$. Hence, the set $\{e_F\,|\,F=\{j_1<j_2<\ldots<j_i\}\}\subseteq [\mu]:=\{1,2,\ldots,\mu\}$, where $e_F = e_{j_1}\wedge e_{j_2}\wedge \cdots \wedge e_{j_i}$ form a basis of $T_i$. The differential 
$\delta_i:T_i \rightarrow T_{i-1}$ is defined as $\delta_i(e_F) =\sum_{k=1}^i(-1)^k\frac{{\rm lcm}(m_F)}{{\rm lcm}(m_{F\setminus \{j_k\}})}e_{F\setminus \{j_k\}}$, where ${\rm lcm}(m_F)={\rm lcm}(m_{j_1},\ldots,m_{j_i})$. The complex $\mathbb T.$ gives a (possibly non-minimal) graded free resolution of $R/I$.  We thank K. Kimura for suggesting us the following statement.
\begin{Proposition}\label{eq}
Let $\H',\H$ be hypergraphs with $\H=\H'\cup G$ where $G=\{i_1,\ldots,i_r\}$. If $\{i_j\}\in \H'$ for all $j$, then $\pd(\H')=\pd(\H)$.
\end{Proposition}

\begin{proof}
If $G\in \H'$, then $\H'=\H$ and the statement is trivial. We may then assume $G\notin \H'$, which, by assumption, implies $r\geq 2$.
Let $I'=(m_1',\ldots,m_{\mu}')\subseteq R'$ be an ideal associated to $\H'$, let $b$ be a new variable, and let $I=(m_1,\ldots,m_{\mu})\subseteq R=R'[b]$, where for each $l$ we define
$$m_l=\left\{ \begin{array}{ll}
m_l', &\mbox{ if }l\notin G\\
m_l'b, & \mbox{ if }l\in G
\end{array}\right.$$
 By construction, $I$ is associated to $\H$. Then, by Corollary \ref{same}, we have $\pd(\H)=\pd(R/I)$ and $\pd(\H')=\pd(R'/I')$.
Set $p=\pd(\H)$, by Lemma \ref{loc}, we only need to show $p\leq \pd(\H')$.  Let $M'$ be the unique homogeneous maximal ideal of $R'$ and let $M=M'R+bR$ be the homogeneous maximal ideal of $R$ and let $\mathbb T.$ and $\mathbb T.'$ be the Taylor's complexes of $I$ and $I'$ respectively.

 By definition of $p$, there exists a subset $\overline{F}\subseteq [\mu]$ such that $\delta_p(e_{\overline{F}})\in MT_{p-1}$ or, equivalently, $\frac{{\rm lcm}(m_{\overline{F}})}{{\rm lcm}(m_{{\overline{F}}\setminus \{j_k\}})}\in M$ for every $k$. To prove $\pd(\H')\geq p$ it suffices to show that $\delta_p'(e_{\overline{F}}')\in M'T_{p-1}'$.

If either $G\cap \overline{F}=\emptyset$ or $|G\cap \overline{F}|\geq 2$, then, by definition of the $m_l$, the variable $b$ does not divide the monomial $\frac{{\rm lcm}(m_{\overline{F}}')}{{\rm lcm}(m_{{\overline{F}}\setminus \{j_k\}}')}=\frac{{\rm lcm}(m_{\overline{F}})}{{\rm lcm}(m_{{\overline{F}}\setminus \{j_k\}})}$ of $M$, for every $k$. This proves that 
$\frac{{\rm lcm}(m_{\overline{F}}')}{{\rm lcm}(m_{{\overline{F}}\setminus \{j_k\}}')}\in M'$, whence
$$\delta_p'(e_{\overline{F}}') =\sum_{k=1}^p(-1)^k\frac{{\rm lcm}(m_{\overline{F}}')}{{\rm lcm}(m_{{\overline{F}}\setminus \{j_k\}}')}e_{{\overline{F}}\setminus \{j_k\}}'\in M'T_{p-1}'.$$

Hence, we may assume $\overline{F}\cap G=\{i_h\}$ for some $1\leq h \leq r$. Write $\overline{F}=\{i_h,j_1,\ldots,j_{p-1}\}$, where the $j_k\notin G$, then, for any $k$, we have
$\frac{{\rm lcm}(m_{\overline{F}}')}{{\rm lcm}(m_{\overline{F}\setminus \{j_k\}}')}=\frac{b\cdot {\rm lcm}(m_{\overline{F}})}{b\cdot {\rm lcm}(m_{\overline{F}\setminus \{j_k\}})}=\frac{{\rm lcm}(m_{\overline{F}})}{{\rm lcm}(m_{\overline{F}\setminus \{j_k\}})}\in M'$. 

Since $\delta_p'(e_{\overline{F}}') =\sum_{k=1}^{p-1}(-1)^k\frac{{\rm lcm}(m_{\overline{F}}')}{{\rm lcm}(m_{\overline{F}\setminus \{j_k\}}')}e_{\overline{F}\setminus \{j_k\}}' + (-1)^p\frac{{\rm lcm}(m_{\overline{F}}')}{{\rm lcm}(m_{\overline{F}\setminus \{i_h\}}')}e_{\overline{F}\setminus \{i_h\}}'$, it only remains to show that $\frac{{\rm lcm}(m_{\overline{F}}')}{{\rm lcm}(m_{\overline{F}\setminus \{i_h\}}')}\in M'$. 
By assumption, $i_h$ is a closed vertex of $\H'$, hence there exists a variable $a_{i_h}$ of $R'$ such that $a_{i_h}$ divides only $m_{i_h}'$ and does not divide any other minimal generator of $I'$. Then, $\frac{{\rm lcm}(m_{\overline{F}}')}{{\rm lcm}(m_{\overline{F}\setminus \{i_h\}}')}\in a_{i_h}R'\subseteq M'$, which concludes the proof.
\end{proof}

The next lemma allows us to control the projective dimension when we remove from $\H$ a closed vertex whose neighbors are closed. We thank K. Kimura for suggesting us the following proof (which allowed us to remove an addition assumption on $\H$ in our original statement).

\begin{Lemma}\label{red}
Let $\H$ be a hypergraph. If $\{1\}\in \H$ and all its neighbors are closed vertices, then $\pd(\H)=\pd(\H_1)+1$.
\end{Lemma}

\begin{proof}
By Lemma \ref{ref} we have
$\pd(\H)={\rm max}\{\pd(\H_1),\pd(\Q_1)+1\}$, 
hence it suffices to show that $\pd(\Q_1)=\pd(\H_1)$. Since $\Q_1\subseteq \H_1$ and all the neighbors of $1$ are closed, then the vertices of the faces of $\H_1$ that are not in $\Q_1$ are all closed. Thus, by iterated use of Proposition \ref{eq}, one has $\pd(\Q_1)=\pd(\H_1)$. 
\end{proof}

The next example, shown in Figure~\ref{F1}, illustrates the use of Lemma \ref{red}.
\begin{Example}\label{1}
Let $\mathcal{H}=\{\{1\},\{1,2\},\{2\},\{2,3\},\{3,4\},\{4,5\},\{5\},\{5,6\},\{6,7\},\{7\}\}$  and $\mathcal{H}_{1}=\{\{2\},\{2,3\},\{3,4\},\{4,5\},\{5\},\{5,6\},\{6,7\},\{7\}\}$, then $\pd(\H)=\pd(\H_1)+1$.
\begin{figure}[h] 
\begin{tikzpicture}[thick, scale=0.7]

\shade [shading=ball, ball color=black] (0,0) circle (.1) node [left] {$\mathcal{H}:$  } ; 
\shade [shading=ball, ball color=black] (1,0) circle (.1) ;
\draw  [shape=circle] (2,0) circle (.1);
\draw  [shape=circle] (3,0) circle (.1) ;
\shade [shading=ball, ball color=black]  (4,0) circle (.1) ;
\draw  [shape=circle] (5,0) circle (.1) ;
\shade [shading=ball, ball color=black] (6,0) circle (.1);
\draw [line width=1.2pt  ] (0,0)--(0.9,0);
\draw [line width=1.2pt  ] (1.1,0)--(1.9,0)  ;
\draw [line width=1.2pt  ] (2.1,0)--(2.9,0)  ;
\draw [line width=1.2pt  ] (3.1,0)--(3.9,0)  ;
\draw [line width=1.2pt  ] (4.1,0)--(4.9,0)  ;
\draw [line width=1.2pt  ] (5.1,0)--(5.9,0)  ;

\shade [shading=ball, ball color=black] (1,-1) circle (.1) ;
\draw  [shape=circle] (2,-1) circle (.1) ;
\draw  [shape=circle] (3,-1) circle (.1);
\shade [shading=ball, ball color=black]  (4,-1) circle (.1) ;
\draw  [shape=circle] (5,-1) circle (.1) ;
\shade [shading=ball, ball color=black] (6,-1) circle (.1) ;
\path (0,-1)--(0,-1) node [left] {$\mathcal{H}_{1}:$  } ; 
\draw [line width=1.2pt  ] (1,-1)--(1.9,-1);
\draw [line width=1.2pt  ] (2.1,-1)--(2.9,-1)  ;
\draw [line width=1.2pt  ] (3.1,-1)--(3.9,-1)  ;
\draw [line width=1.2pt  ] (4.1,-1)--(4.9,-1)  ;
\draw [line width=1.2pt  ] (5.1,-1)--(5.9,-1)  ;

\end{tikzpicture}

\caption{}\label{F1}
\end{figure}
\end{Example}

We now introduce the first class of hypergraphs studied in this paper.
\begin{Definition}\label{str}
A hypergraph $\H$ with $V=[\mu]$ is a {\em string} if $\{i,i+1\}\in \H$ for all $i=1,\ldots,\mu-1$ and 
the only faces containing the vertex $i$ are  $\{i-1,i\},\{i,i+1\}$ and, possibly, $\{i\}$.

The vertices $1$ and $\mu$ are called the {\em endpoints} of the string $\H$.
\end{Definition}
From the definition of separated hypergraph it follows that the endpoints of a string hypergraph are closed vertices, i.e. $\{1\}\in \H$ and $\{\mu\}\in \H$. A hypergraph $\H$ is called a {\em string of opens} if $\H$ is a string hypergraph with at least three vertices such that the only closed vertices of $\H$ are its endpoints. 
For instance, the two hypergraphs $\H'=\{\{1\},\{1,2\},\{2,3\},\{3\}\},$ and $\H''=\{\{1\}, \{1,2\}, \{2,3\},$ $\{3,4\}, \{4,5\}, \{5,6\}, \{6,7\}, \{7\} \} $ are strings of opens, whereas the hypergraphs $\H$ and $\H'$ of Example \ref{1} are string hypergraphs that are not string of opens (because they contain four and three closed vertices, respectively). Note that there is a bijective correspondence between strings of opens and string edge ideals.

For the reader's convenience, the vertex $1$ of a string will always denote one of the endpoints of $\H$. 
Similarly, if $\mu\geq 2$, then the vertex $2$ of $\H$ is the only neighbor of $1$, if $\mu\geq 3$, then the vertex $3$ is the other neighbor of $2$, etc. If $\H$ is a string with $\mu\geq 2$ vertices, then, following Notation \ref{not}, $\H_1$ is the string obtained by ''removing'' the endpoint $1$ from $\H$; the vertex $2$ is then an endpoint of $\H_1$.

\begin{Lemma}\label{red2}
Let $\H$ be any string hypergraph. Then, one has the inequalities
\begin{itemize}
\item[(i)]  $\pd(\H_1)\leq \pd(\H)\leq \pd(\H_1)+1$, and
\item[(ii)] $\pd(\H)\leq \pd(\H_2)+2$. 
\end{itemize}
\end{Lemma}

\begin{proof}
(i) If $\{2\}\in \H$, 
then, by Lemma \ref{red}, we have $\pd(\H)=\pd(\H_1)+1$ and the statement follows trivially. We may then assume $\{2\}\notin \H$. Now, the inequality $\pd(\H_1)\leq \pd(\H)$ follows from Lemma \ref{ref}, so we only need to show $\pd(\H)\leq \pd(\H_1)+1$.  Let $\overline{\H}=\H\cup \{2\}$ and note that, since $2$ is a closed vertex of $\overline{\H}$, by Lemma \ref{ref} we have $\pd(\overline{\H})=\pd(\overline{\H}_1)+1$. Since $\overline{\H}_1=\H_1$, we have $\pd(\overline{\H})=\pd(\H_1)+1$. Also, by Corollary \ref{compare}, we have $\pd(\H)\leq \pd(\overline{\H})$. Hence we obtain
$$\pd(\H)\leq \pd(\overline{\H})=\pd(\H_1)+1.$$
Assertion (ii) follows by applying assertion (i) twice: $\pd(\H)\leq \pd(\H_1)+1\leq \pd(\H_2)+2$. 
\end{proof}

We now prove that, if the neighbor of an endpoint is open, ``removing'' the last 3 vertices from that end of the string makes the projective dimension drop exactly by 2 units.
\begin{Proposition}\label{red3}
Let $\H$ be a string with $\mu\geq 3$ vertices. If $\{2\}\notin \H$, then $\pd(\H)=\pd(\H_3)+2$.
\end{Proposition}

\begin{proof}
Recall that the vertex $2$ is the neighbor of the endpoint $1$ of $H$, and $3$ is the other neighbor of $2$. 
By Lemma \ref{ref} we have $\pd(\H)={\rm max}\{\pd(\H_1),\pd(\Q_1)+1\}$. Moreover, since $\{2\}\notin \H$, it is easily seen that $\Q_1=\H_3\cup\{v\}$ is the disjoint union of $\H_3$ and a vertex $v$, hence $\pd(\Q_1)=\pd(\H_3)+1$. Then we have
$$\pd(\H)={\rm max}\{\pd(\H_1),\pd(\H_3)+2\}=\pd(\H_3)+2,$$
where the last equality follows because $\pd(\H_1)\leq \pd(\H_3)+2$ by Lemma \ref{red2}.(ii).
\end{proof}

Combining together Lemma \ref{red} and Proposition \ref{red3} we obtain a recursive way of computing projective dimensions, see also Algorithm \ref{alg}.
\begin{Remark}\label{proc}
Let $\H$ be a string hypergraph. If the vertex $2$ is closed, then $\pd(\H)=\pd(\H_1)+1$; if the vertex $2$ is open, then $\pd(\H)=\pd(\H_3)+2$. In either way, to compute $\pd(\H)$ one has to compute the projective dimension of a string with a strictly smaller number of vertices. After a finite number of iterations of this procedure, one obtains the projective dimension of $\H$.
\end{Remark}

Example \ref{ex1}, which is shown in Figure ~\ref{F2}, illustrates the procedure.
\begin{Example}\label{ex1}
Let $\mathcal{H}$  be the hypergraph of Example \ref{1}. Then $\pd(\mathcal{H})=5$.
\end{Example}

\begin{proof}
By Lemma \ref{red} we have $\pd(\H)=\pd(\H_1)+1$, and 
by Proposition \ref{red3}, $\pd(\mathcal{H}_{1})=\pd(\mathcal{H}_{4})+2$. 
By Proposition \ref{red3}, we have $\pd(\mathcal{H}_{4})=\pd(\H_7)+2=0+2=2$ (because $\H_7$ is the empty hypergraph), then $\pd(\mathcal{H})=1+2+2=5$.
\begin{figure}[h] \label{F2}
\begin{tikzpicture}[thick, scale=0.7]

\shade [shading=ball, ball color=black] (0,0) circle (.1) node [left] {$\mathcal{H}:$  } ; 
\shade [shading=ball, ball color=black] (1,0) circle (.1) ;
\draw  [shape=circle] (2,0) circle (.1);
\draw  [shape=circle](3,0) circle (.1) ;
\shade [shading=ball, ball color=black] (4,0) circle (.1) ;
\draw  [shape=circle] (5,0) circle (.1) ;
\shade [shading=ball, ball color=black] (6,0) circle (.1);
\draw [line width=1.2pt  ] (0,0)--(0.9,0);
\draw [line width=1.2pt  ] (1.1,0)--(1.9,0)  ;
\draw [line width=1.2pt  ] (2.1,0)--(2.9,0)  ;
\draw [line width=1.2pt  ] (3.1,0)--(3.9,0)  ;
\draw [line width=1.2pt  ] (4.1,0)--(4.9,0)  ;
\draw [line width=1.2pt  ] (5.1,0)--(5.9,0)  ;

\shade [shading=ball, ball color=black] (1,-1) circle (.1) ;
\draw  [shape=circle] (2,-1) circle (.1) ;
\draw  [shape=circle] (3,-1) circle (.1);
\shade [shading=ball, ball color=black] (4,-1) circle (.1) ;
\draw  [shape=circle] (5,-1) circle (.1) ;
\shade [shading=ball, ball color=black] (6,-1) circle (.1) ;
\path (0,-1)--(0,-1) node [left] {$\mathcal{H}_{1}:$  } ; 
\draw [line width=1.2pt  ] (1.1,-1)--(1.9,-1)  ;
\draw [line width=1.2pt  ] (2.1,-1)--(2.9,-1)  ;
\draw [line width=1.2pt  ] (3.1,-1)--(3.9,-1)  ;
\draw [line width=1.2pt  ] (4.1,-1)--(4.9,-1)  ;
\draw [line width=1.2pt  ] (5.1,-1)--(5.9,-1);

\shade [shading=ball, ball color=black] (4,-2) circle (.1);
\draw  [shape=circle] (5,-2) circle (.1) ;
\shade [shading=ball, ball color=black] (6,-2) circle (.1);
\draw [line width=1.2pt  ] (4.1,-2)--(4.9,-2);
\draw [line width=1.2pt  ] (5.1,-2)--(5.9,-2)  ;
\path (0,-2)--(0,-2) node [left] {$\mathcal{H}_{4}:$  } ; 
 
\end{tikzpicture}
\caption{}
\end{figure}
\end{proof}

\section{Projective Dimension of Strings}

Lemma \ref{red} and Proposition \ref{red3} provide a simple recursive procedure to compute the projective dimension of any square-free monomial ideal associated to any string hypergraph $\H$. However, it is desirable to have a {\em closed} formula for $\pd(\H)$ in terms of few combinatorial data of $\H$. This is achieved in the main result of this section, Theorem \ref{string}. In fact, we show that the projective dimension of $I(\H)$ is uniquely determined by (1) the total number of vertices in the string, (2) an invariant $b(\H)$ depending on the strings of open vertices of $\H$ (see discussion after Remark \ref{2spec}), and (3) a combinatorial invariant of the hypergraph, the {\it modularity} of $\H$, which we now introduce.
\bigskip

To define the modularity,  we first need to isolate a special class of strings. Let $\H$ be a string hypergraph containing exactly $s\geq 2$ strings of opens. We number the strings of opens from one endpoint to the other one 
and we let $n_i$ be the number of open vertices in the $i$-th string of opens.
We say $\H$ is a {\em 2-special configuration} if \begin{itemize}
\item $\H$ does not contain two adjacent closed vertices, and
\item $n_1\equiv n_s\equiv 1$ mod (3), and $n_i\equiv 2$ mod (3) for every $1<i<s$.
\end{itemize}

In Figure \ref{2S} we provide a few examples of 2-special configurations. The hypergraph $\H$ has $n_1=n_2=1$; the hypergraph $\H'$ has  $n_1=4$, $n_2=5$, $n_3=1$; the hypergraph $\H''$ has $n_1=1$, $n_2=5$, $n_3=2$ and $n_4=1$.
\begin{figure}[h] 

\begin{tikzpicture}[thick, scale=0.7]

\shade [shading=ball, ball color=black] (0,0) circle (.1) node [left] {$\H:$  } ; 
\draw  [shape=circle] (1,0) circle (.1) ;
\shade [shading=ball, ball color=black] (2,0) circle (.1);
\draw  [shape=circle] (3,0) circle (.1) ;
\shade [shading=ball, ball color=black] (4,0) circle (.1) ;
\draw [line width=1.2pt  ] (0,0)--(0.9,0);
\draw [line width=1.2pt  ] (1.1,0)--(1.9,0)  ;
\draw [line width=1.2pt  ] (2.1,0)--(2.9,0)  ;
\draw [line width=1.2pt  ] (3.1,0)--(3.9,0)  ;

\shade [shading=ball, ball color=black] (0,-1) circle (.1) node [left] {$\H':$  } ; 
\draw  [shape=circle] (1,-1) circle (.1) ;
\draw  [shape=circle] (2,-1) circle (.1) ;
\draw  [shape=circle] (3,-1) circle (.1) ;
\draw  [shape=circle] (4,-1) circle (.1) ;
\shade [shading=ball, ball color=black] (5,-1) circle (.1);
\draw  [shape=circle] (6,-1) circle (.1) ;
\draw  [shape=circle] (7,-1) circle (.1) ;
\draw  [shape=circle] (8,-1) circle (.1) ;
\draw  [shape=circle] (9,-1) circle (.1) ;
\draw  [shape=circle] (10,-1) circle (.1) ;
\shade [shading=ball, ball color=black] (11,-1) circle (.1) ;
\draw  [shape=circle] (12,-1) circle (.1) ;
\shade [shading=ball, ball color=black] (13,-1) circle (.1) ;
\draw [line width=1.2pt  ] (0,-1)--(0.9,-1);
\draw [line width=1.2pt  ] (1.1,-1)--(1.9,-1)  ;
\draw [line width=1.2pt  ] (2.1,-1)--(2.9,-1)  ;
\draw [line width=1.2pt  ] (3.1,-1)--(3.9,-1)  ;
\draw [line width=1.2pt  ] (4.1,-1)--(4.9,-1)  ;
\draw [line width=1.2pt  ] (5.1,-1)--(5.9,-1)  ;
\draw [line width=1.2pt  ] (6.1,-1)--(6.9,-1)  ;
\draw [line width=1.2pt  ] (7.1,-1)--(7.9,-1)  ;
\draw [line width=1.2pt  ] (8.1,-1)--(8.9,-1)  ;
\draw [line width=1.2pt  ] (9.1,-1)--(9.9,-1)  ;
\draw [line width=1.2pt  ] (10.1,-1)--(10.9,-1)  ;
\draw [line width=1.2pt  ] (11.1,-1)--(11.9,-1)  ;
\draw [line width=1.2pt  ] (12.1,-1)--(12.9,-1)  ;

\shade [shading=ball, ball color=black] (0,-2) circle (.1) node [left] {$\H'':$  } ; 
\draw  [shape=circle] (1,-2) circle (.1) ;
\shade [shading=ball, ball color=black] (2,-2) circle (.1);
\draw  [shape=circle] (3,-2) circle (.1) ;
\draw  [shape=circle] (4,-2) circle (.1) ;
\draw  [shape=circle] (5,-2) circle (.1) ;
\draw  [shape=circle] (6,-2) circle (.1) ;
\draw  [shape=circle] (7,-2) circle (.1) ;
\shade [shading=ball, ball color=black] (8,-2) circle (.1) ;
\draw  [shape=circle] (9,-2) circle (.1) ;
\draw  [shape=circle] (10,-2) circle (.1) ;
\shade [shading=ball, ball color=black] (11,-2) circle (.1) ;
\draw  [shape=circle] (12,-2) circle (.1) ;
\shade [shading=ball, ball color=black] (13,-2) circle (.1) ;

\draw [line width=1.2pt  ] (0,-2)--(0.9,-2);
\draw [line width=1.2pt  ] (1.1,-2)--(1.9,-2)  ;
\draw [line width=1.2pt  ] (2.1,-2)--(2.9,-2)  ;
\draw [line width=1.2pt  ] (3.1,-2)--(3.9,-2)  ;
\draw [line width=1.2pt  ] (4.1,-2)--(4.9,-2)  ;
\draw [line width=1.2pt  ] (5.1,-2)--(5.9,-2)  ;
\draw [line width=1.2pt  ] (6.1,-2)--(6.9,-2)  ;
\draw [line width=1.2pt  ] (7.1,-2)--(7.9,-2)  ;
\draw [line width=1.2pt  ] (8.1,-2)--(8.9,-2)  ;
\draw [line width=1.2pt  ] (9.1,-2)--(9.9,-2)  ;
\draw [line width=1.2pt  ] (10.1,-2)--(10.9,-2)  ;
\draw [line width=1.2pt  ] (11.1,-2)--(11.9,-2)  ;
\draw [line width=1.2pt  ] (12.1,-2)--(12.9,-2)  ;

\end{tikzpicture}

\caption{}\label{2S}
\end{figure}

\noindent Two 2-special configurations contained in the same string hypergraph are {\it disjoint} if they do not have any {\it open} vertex in common; however, they may share a closed vertex.
\begin{Definition}\label{modularity}
The {\rm modularity} $M(\H)$ of a string hypergraph $\H$ is the maximal number of pairwise disjoint 2-special configurations contained in $\H$. 
\end{Definition}
For instance, the string  $$\H=\{\{1\},\{1,2\},\{2,3\},\{3\},\{3,4\},\{4,5\},\{5\},\{5,6\},\{6,7\},\{7\}, \{7,8\}, \{8,9\}, \{9\}\}$$ has modularity $M(\H)=2$ because the two 2-special configurations $\{\{1\},\{1,2\},\{2,3\},\{3\}$, $\{3,4\},\{4,5\},\{5\}\}$ and $\{\{5\},\{5,6\},\{6,7\},\{7\}, \{7,8\}, \{8,9\}\}$ only share the closed vertex $5$. On the other hand, the string
$\H'=\{\{1\},\{1,2\},\{2,3\},\{3\},\{3,4\},\{4,5\},\{5\},\{5,6\},\{6,7\},\{7\}\}$
has modularity $M(\H')=1$. Indeed, $\H'$ does {\it not} contain two disjoint 2-special configurations, because  $\{\{1\},\{1,2\},\{2,3\},\{3\},\{3,4\},\{4,5\},\{5\}\}$ and $\{\{3\},\{3,4\},\{4,5\},\{5\},\{5,6\},\{6,7\},\{7\}\}$ are the only 2-special configurations contained in $\H'$, and they share the open vertex $4$.

The following remark follows immediately from the definition.
\begin{Remark}\label{2spec}
If a string of opens $\OO$ contains $n$ open vertices with $n\equiv 0$ modulo 3, then $\OO$ is not contained in any 2-special configuration.

Also, a 2-special configuration in $\H$  that is disjoint from every other 2-special configuration in $\H$ lies in every maximal set of 2-special configurations of $\H$.
\end{Remark}

Let $\H$ be a string hypergraph. We set:
$$\begin{array}{ll}
\mu(\H) & = \mbox{  number of vertices in }\H;\\
s(\H) & = \mbox{ number of strings of opens contained in }\H;\\
M(\H) & = \mbox{ modularity of } \H;\\
\end{array}$$
\noindent
 Moreover, if $\OO_1,\ldots,\OO_s$ are all the strings of opens in $\H$ 
 we set 
$n_i=n_i(\H)$  to be the number of open vertices in $\OO_i$. Here we introduce the invariant
$$b(\H)=s(\H)+\sum\limits_{i=1}^s\left\lfloor\frac{n_i-1}{3}\right\rfloor.$$
The above invariants can be easily computed, as the following example illustrates.
\begin{Example}\label{ex3}
Let $\H=\{\{1\},\{1,2\},\{2,3\},\{3,4\},\{4,5\},\{5,6\},\{6\},\{6,7\},\{7,8\},\{8\},\{8,9\},$ $\{9\},
\{9,10\},\{10,11\},\{11\}\},$ then 
$$\mu(\mathcal{H})=11,\; s(\mathcal{H})=3,\;M(\mathcal{H})=1,\;\mbox{ and } b(\H)=3+\left\lfloor\frac{4-1}{3}\right\rfloor+ \left\lfloor\frac{1-1}{3}\right\rfloor+ \left\lfloor\frac{1-1}{3}\right\rfloor=4$$
Example \ref{ex3} is shown in Figure~\ref{F3}.
\begin{figure}[h] 

\begin{tikzpicture}[thick, scale=0.7]

\shade [shading=ball, ball color=black] (0,0) circle (.1) node [left] {$\mathcal{H}:$  } ; 
\draw  [shape=circle] (1,0) circle (.1) ;
\draw  [shape=circle] (2,0) circle (.1) ;
\draw  [shape=circle] (3,0) circle (.1) ;
\draw  [shape=circle] (4,0) circle (.1) ;
\shade [shading=ball, ball color=black] (5,0) circle (.1) ;
\draw  [shape=circle] (6,0) circle (.1) ;
\shade [shading=ball, ball color=black] (7,0) circle (.1);
\shade [shading=ball, ball color=black] (8,0) circle (.1);
\draw  [shape=circle] (9,0) circle (.1) ;
\shade [shading=ball, ball color=black] (10,0) circle (.1);
\draw [line width=1.2pt  ] (0,0)--(0.9,0);
\draw [line width=1.2pt  ] (1.1,0)--(1.9,0)  ;
\draw [line width=1.2pt  ] (2.1,0)--(2.9,0)  ;
\draw [line width=1.2pt  ] (3.1,0)--(3.9,0)  ;
\draw [line width=1.2pt  ] (4.1,0)--(4.9,0)  ;
\draw [line width=1.2pt  ] (5.1,0)--(5.9,0)  ;
\draw [line width=1.2pt  ] (6.1,0)--(6.9,0)  ;
\draw [line width=1.2pt  ] (7.1,0)--(7.9,0)  ;
\draw [line width=1.2pt  ] (8.1,0)--(8.9,0)  ;
\draw [line width=1.2pt  ] (9.1,0)--(9.9,0)  ;

\end{tikzpicture}

\caption{}\label{F3}
\end{figure}
\end{Example}

We are now ready to state the main result of this section.
\begin{Theorem}\label{string}
If $\H$ is a string hypergraph, then $\pd(\H)=\mu(\H)-b(\H)+M(\H)$.
\end{Theorem}

As an example, if $\H$ is the hypergraph of Example \ref{ex3}, then $\pd(\H)=11-4+1=8$. The proof of Theorem \ref{string} is postponed to the end of this section. Here, we wish to make a few observations.
The first one is that the class of ideals of Theorem \ref{string} is not covered by previous results of Schweig-Dao. Indeed, the next simple example shows that there are string hypergraphs whose corresponding clutters are not edgewise dominant (cf. \cite{DS}) and for which the formula $\pd(\mathcal{C})=V(\mathcal{C})-i(\mathcal{C})$ provided in \cite{DS} does not hold.

\begin{Example}\label{DS2}
Let $\mathcal{H}$ and $\mathcal{C}$ be the hypergraph and clutter of the ideal $I=(ab,bc,cde,ef,fg)$, see Figure~\ref{Cstring}. Then $\pd(\mathcal{H})=5-2+1\,>\,V(\mathcal{C})-i(\mathcal{C})=7-4$.
\begin{figure}[h] 

\begin{tikzpicture}[thick, scale=0.7]

\shade [shading=ball, ball color=black] (0,0) circle (.1) node [left] {$\mathcal{H}:$  } ; 
\draw  [shape=circle] (1,0) circle (.1) ;
\shade [shading=ball, ball color=black] (2,0) circle (.1) ;
\draw  [shape=circle] (3,0) circle (.1) ;
\shade [shading=ball, ball color=black] (4,0) circle (.1) ;
\draw [line width=1.2pt  ] (0,0)--(0.9,0);
\draw [line width=1.2pt  ] (1.1,0)--(1.9,0)  ;
\draw [line width=1.2pt  ] (2.1,0)--(2.9,0)  ;
\draw [line width=1.2pt  ] (3.1,0)--(3.9,0)  ;

\shade [shading=ball, ball color=black] (6,0) circle (.1) node [left] {$\mathcal{C}:$  } ; 
\shade [shading=ball, ball color=black] (7,0) circle (.1) ;
\shade [shading=ball, ball color=black] (8,0) circle (.1) ;
\shade [shading=ball, ball color=black] (9,0) circle (.1) ;
\shade [shading=ball, ball color=black] (10,0) circle (.1) ;
\shade [shading=ball, ball color=black] (11,0) circle (.1) ;
\shade [shading=ball, ball color=black] (8.5,-1) circle (.1) ;
\path [pattern=north east lines]   (8,0)--(9,0)--(8.5,-1)--cycle;
\draw [line width=1.2pt  ] (6,0)--(6.9,0);
\draw [line width=1.2pt  ] (7.1,0)--(7.9,0)  ;
\draw [line width=1.2pt  ] (9.1,0)--(9.9,0);
\draw [line width=1.2pt  ] (10.1,0)--(10.9,0)  ;
\end{tikzpicture}

\caption{}\label{Cstring}
\end{figure}
\end{Example}
The second observation is that Theorem \ref{string} also allows one to compute the projective dimension of $I(\H)$ when $\H$ is the disjoint union of a finite number of string hypergraphs.

The next remark is that, in general, permuting the strings of opens of $\H$ has an impact on $\pd(\H)$. In fact, the particular order of the strings of opens does not affect $b(\H)$, but it may modify the modularity, which, in turn, impacts the projective dimension of $I(\H)$. For instance, set $\H'=\{\{1\},\{1,2\},\{2,3\},\{3\},\{3,4\},\{4,5\},\{5,6\},\{6,7\},\{7\},\{7,8\},\{8,9\},\{9\}\}$ and $\H''=\{\{1\},\{1,2\},\{2,3\},\{3\},\{3,4\},\{4,5\},\{5\},\{5,6\},\{6,7\},\{7,8\},\{8,9\},\{9\}\}$, and note that $\H''$ can be obtained by permuting  
the strings of opens of $\H'$. The hypergraphs $\H'$ and $\H''$ are shown in Figure.~\ref{exchange}. Note that $\pd(\H')=6$, whereas $\pd(\H'')=7$; this difference depends on the fact that, by Remark \ref{2spec}, $\H'$ has modularity $0$, whereas $M(\H'')=1$. 
\begin{figure}[h] 
\begin{tikzpicture}[thick, scale=0.7]

\shade [shading=ball, ball color=black] (0,0) circle (.1) node [left] {$\H':$  } ; 
\draw  [shape=circle] (1,0) circle (.1) ;
\shade [shading=ball, ball color=black] (2,0) circle (.1) ;
\draw  [shape=circle] (3,0) circle (.1) ;
\draw  [shape=circle] (4,0) circle (.1) ;
\draw  [shape=circle] (5,0) circle (.1) ;
\shade [shading=ball, ball color=black] (6,0) circle (.1);
\draw  [shape=circle] (7,0) circle (.1) ;
\shade [shading=ball, ball color=black] (8,0) circle (.1);
\draw [line width=1.2pt  ] (0,0)--(0.9,0);
\draw [line width=1.2pt  ] (1.1,0)--(1.9,0)  ;
\draw [line width=1.2pt  ] (2.1,0)--(2.9,0)  ;
\draw [line width=1.2pt  ] (3.1,0)--(3.9,0)  ;
\draw [line width=1.2pt  ] (4.1,0)--(4.9,0)  ;
\draw [line width=1.2pt  ] (5.1,0)--(5.9,0)  ;
\draw [line width=1.2pt  ] (6.1,0)--(6.9,0)  ;
\draw [line width=1.2pt  ] (7.1,0)--(7.9,0)  ;

\shade [shading=ball, ball color=black] (10,0) circle (.1) node [left] {$\H'':$  } ; 
\draw  [shape=circle] (11,0) circle (.1) ;
\shade [shading=ball, ball color=black] (12,0) circle (.1) ;
\draw  [shape=circle] (13,0) circle (.1) ;
\shade [shading=ball, ball color=black] (14,0) circle (.1);
\draw  [shape=circle] (15,0) circle (.1) ;
\draw  [shape=circle] (16,0) circle (.1) ;
\draw  [shape=circle] (17,0) circle (.1) ;
\shade [shading=ball, ball color=black] (18,0) circle (.1);
\draw [line width=1.2pt  ] (10,0)--(10.9,0);
\draw [line width=1.2pt  ] (11.1,0)--(11.9,0)  ;
\draw [line width=1.2pt  ] (12.1,0)--(12.9,0)  ;
\draw [line width=1.2pt  ] (13.1,0)--(13.9,0)  ;
\draw [line width=1.2pt  ] (14.1,0)--(14.9,0)  ;
\draw [line width=1.2pt  ] (15.1,0)--(15.9,0)  ;
\draw [line width=1.2pt  ] (16.1,0)--(16.9,0)  ;
\draw [line width=1.2pt  ] (17.1,0)--(17.9,0)  ;

\end{tikzpicture}
\caption{}\label{exchange}
\end{figure}

In contrast to the above examples, the next result shows that, in some cases, permutations of the strings of opens can be performed without modifying the projective dimension. 
\begin{Corollary}\label{order}
Let $\H$ and $\H'$ be string hypergraphs both having $\mu$ vertices, containing exactly $s$ strings of opens, which have $n_1,\ldots,n_s$ open vertices. If all the $n_i$ except, possibly, one satisfy $n_i\equiv \!\!\!\!\not\;\,1$ mod 3, then $\pd(\H)=\pd(\H')$.
\end{Corollary}
\begin{proof}
The assumptions yield $M(\H)=M(\H')=0$. Now, Theorem \ref{string} concludes the proof.
\end{proof}

We provide an example illustrating Corollary \ref{order}.
\begin{Example}\label{mod}
Let $\H$ and $\H'$ be string hypergraphs both having $\mu=15$ total vertices, and $s=3$ strings of opens that have $1$, $3$ and $5$ open vertices, respectively. Then $\pd(\H)=\pd(\H')=11$. 
Possible figures of $\H$ and $\H'$ are shown below, in Figure~\ref{Switch}.
\begin{figure}[h] 

\begin{tikzpicture}[thick, scale=0.7]
\shade [shading=ball, ball color=black] (0,0) circle (.1) node [left] {$\H:$  } ; 
\draw  [shape=circle] (1,0) circle (.1) ;
\shade [shading=ball, ball color=black] (2,0) circle (.1) ;
\draw  [shape=circle] (3,0) circle (.1) ;
\draw  [shape=circle] (4,0) circle (.1) ;
\draw  [shape=circle] (5,0) circle (.1) ;
\shade [shading=ball, ball color=black] (6,0) circle (.1) ;
\shade [shading=ball, ball color=black] (7,0) circle (.1) ;
\shade [shading=ball, ball color=black] (8,0) circle (.1) ;
\draw  [shape=circle] (9,0) circle (.1) ;
\draw  [shape=circle] (10,0) circle (.1) ;
\draw  [shape=circle] (11,0) circle (.1) ;
\draw  [shape=circle] (12,0) circle (.1) ;
\draw  [shape=circle] (13,0) circle (.1) ;
\shade [shading=ball, ball color=black] (14,0) circle (.1) ;
\draw [line width=1.2pt  ] (0,0)--(0.9,0);
\draw [line width=1.2pt  ] (1.1,0)--(1.9,0)  ;
\draw [line width=1.2pt  ] (2.1,0)--(2.9,0)  ;
\draw [line width=1.2pt  ] (3.1,0)--(3.9,0)  ;
\draw [line width=1.2pt  ] (4.1,0)--(4.9,0)  ;
\draw [line width=1.2pt  ] (5.1,0)--(5.9,0)  ;
\draw [line width=1.2pt  ] (6.1,0)--(6.9,0)  ;
\draw [line width=1.2pt  ] (7.1,0)--(7.9,0)  ;
\draw [line width=1.2pt  ] (8.1,0)--(8.9,0)  ;
\draw [line width=1.2pt  ] (9.1,0)--(9.9,0)  ;
\draw [line width=1.2pt  ] (10.1,0)--(10.9,0)  ;
\draw [line width=1.2pt  ] (11.1,0)--(11.9,0)  ;
\draw [line width=1.2pt  ] (12.1,0)--(12.9,0)  ;
\draw [line width=1.2pt  ] (13.1,0)--(13.9,0)  ;

\shade [shading=ball, ball color=black] (0,-1) circle (.1) node [left] {$\H':$  } ; 
\draw  [shape=circle] (1,-1) circle (.1) ;
\draw  [shape=circle] (2,-1) circle (.1) ;
\draw  [shape=circle] (3,-1) circle (.1) ;
\draw  [shape=circle] (4,-1) circle (.1) ;
\draw  [shape=circle] (5,-1) circle (.1) ;
\shade [shading=ball, ball color=black] (6,-1) circle (.1);
\draw  [shape=circle] (7,-1) circle (.1) ;
\shade [shading=ball, ball color=black] (8,-1) circle (.1) ;
\shade [shading=ball, ball color=black] (9,-1) circle (.1) ;
\draw  [shape=circle] (10,-1) circle (.1) ;
\draw  [shape=circle] (11,-1) circle (.1) ;
\draw  [shape=circle] (12,-1) circle (.1) ;
\shade [shading=ball, ball color=black] (13,-1) circle (.1) ;
\shade [shading=ball, ball color=black] (14,-1) circle (.1) ;
\draw [line width=1.2pt  ] (0,-1)--(0.9,-1);
\draw [line width=1.2pt  ] (1.1,-1)--(1.9,-1)  ;
\draw [line width=1.2pt  ] (2.1,-1)--(2.9,-1)  ;
\draw [line width=1.2pt  ] (3.1,-1)--(3.9,-1)  ;
\draw [line width=1.2pt  ] (4.1,-1)--(4.9,-1)  ;
\draw [line width=1.2pt  ] (5.1,-1)--(5.9,-1)  ;
\draw [line width=1.2pt  ] (6.1,-1)--(6.9,-1)  ;
\draw [line width=1.2pt  ] (7.1,-1)--(7.9,-1)  ;
\draw [line width=1.2pt  ] (8.1,-1)--(8.9,-1)  ;
\draw [line width=1.2pt  ] (9.1,-1)--(9.9,-1)  ;
\draw [line width=1.2pt  ] (10.1,-1)--(10.9,-1)  ;
\draw [line width=1.2pt  ] (11.1,-1)--(11.9,-1)  ;
\draw [line width=1.2pt  ] (12.1,-1)--(12.9,-1)  ;
\draw [line width=1.2pt  ] (13.1,-1)--(13.9,-1)  ;

\end{tikzpicture}

\caption{}\label{Switch}
\end{figure}
\end{Example}
We now apply Theorem \ref{string} to explicitly compute the projective dimension of two simple classes of examples. Note the impact of the modularity in the second class of examples.
\begin{Corollary}\label{examples}
 Let $\H$ be a string hypergraph.
\begin{itemize}
\item[(1)] If $\H$ contains only $s=1$ string of open, let $n_1=n\geq 1$ be its number of open vertices, then 
$$\pd(\H)=\mu(\H)-1-\left\lfloor\frac{n-1}{3}\right\rfloor.$$ 
\item[(2)] Assume $\H$ contains $s=2$ strings of opens, having $n_1,n_2\geq 1$ open vertices, respectively. Let $m\geq 1$ be the number of adjacent closed vertices separating the strings of opens, then 
$$\pd(\H)=\left\{\begin{array}{ll}
\mu(\H)-1-\left\lfloor\frac{n_1-1}{3}\right\rfloor-\left\lfloor\frac{n_2-1}{3}\right\rfloor, \mbox{ if }m=1, \mbox{ and } n_1\equiv n_2\equiv 1 \mbox{ mod } 3\\
{}\\
\mu(\H)-2-\left\lfloor\frac{n_1-1}{3}\right\rfloor-\left\lfloor\frac{n_2-1}{3}\right\rfloor, \mbox{ otherwise }
\end{array}\right.$$
\end{itemize}
\end{Corollary}
\begin{proof}
(1) since $s=1$, we have $M(\H)=0$ and $b(\H)=1+\left\lfloor\frac{n-1}{3}\right\rfloor$, whence the formula follows by Theorem \ref{string}.
For (2), one has $b(\H)=2+\left\lfloor\frac{n_1-1}{3}\right\rfloor+\left\lfloor\frac{n_2-1}{3}\right\rfloor$, whence $\pd(\H)=\mu(\H)-2+\left\lfloor\frac{n_1-1}{3}\right\rfloor+\left\lfloor\frac{n_2-1}{3}\right\rfloor+M(\H)$. Also, since $s=2$, we have $0\leq M(\H)\leq 1$ and $M(\H)=1$ if and only if $m=1$ and $n_1\equiv n_2\equiv 1$ modulo 3. 
\end{proof}

We conclude this section with the proof of Theorem \ref{string}.  We first need a couple of lemmas to compare the combinatorial invariants of $\H$ and $\H_3$. A {\it 1-1 configuration} is a 2-special configuration consisting of $s=2$ strings of opens and for which $n_1=1$.  From the definition it follows that the second string of opens of a 1-1 configuration has a number $n_2$ of open vertices that satisfies $n_2\equiv 1$ modulo 3. An example is given by the hypergraph $\H$ in Figure \ref{2S}. 
\begin{Lemma}\label{H3}
Let $\H$ be a string hypergraph with $\mu(\H)\geq 3$ vertices. Assume $\{2\}\notin \H$. Then $M(\H_3)=M(\H)-1$ if and only if the vertex $2$ is an isolated open vertex contained in a 1-1 special configuration. In all other cases, one has $M(\H_3)=M(\H)$.  
\end{Lemma}
\begin{proof}
Let $\OO_1,\ldots,\OO_s$ be the strings of opens in $\H$ and let $n_1$ be the number of open vertices in the string $\OO_1$ containing the vertex $2$. If $2\leq n_1\leq 3$, then $\OO_1$ is not part of any 2-special configuration of $\H$ and, moreover, $\OO_2,\ldots,\OO_s$ are all the strings of opens of $\H_3$. Hence, $\s$ is a 2-special configuration of $\H$ if and only if $\s$ is a 2-special configuration of $\H_3$, proving that $M(\H_3)=M(\H)$. 

If $n_1>3$, then the strings of opens of $\H_3$ are 
 $\OO_1',\ldots,\OO_s$, where $\OO_1'$ has $n_1-3$ opens. Since $n_1\equiv n_1-3$ modulo 3, we have that $\s=\OO_1\OO_2\ldots\OO_r$ is a 2-special configuration in $\H$ if and only if $\s'=\OO_1'\OO_2\ldots\OO_r$ is a 2-special configuration in $\H_3$. This implies $M(\H_3)=M(\H)$.
 
 We may then assume $n_1=1$. Then, $2$ is an isolated open vertex and $\{3\}\in \H$. If $\{4\}\in \H$, then, similarly to the case $2\leq n_1\leq 3$, $\OO_1$ is not contained in any 2-special configuration, and the strings of opens of $\H_3$ are $\OO_2,\ldots,\OO_s$, showing that $M(\H_3)=M(\H)$. We may then assume the vertex $4$ is open. Let $n_2$ be the number of opens in the string of opens $\OO_2$ containing the vertex $4$. 
  Since $n_1=1$ and the vertex $4$ is open, the strings of opens of $\H_3$ are $\OO_2',\OO_3,\ldots,\OO_s$ where $\OO_2'$ is either the empty string (if $n_2=1$) or $\OO_2'$ has exactly $n_2-1$ open vertices.
 
 First, assume $\OO_1$ is not part of a 1-1 configuration. We show that a maximal set $A$ of disjoint 2-special configurations in $\H$ correspond to a maximal set $A'$ of disjoint 2-special configurations in $\H_3$ with $|A'|=|A|$, i.e., with the same cardinality. If $\OO_1$ is not contained in any 2-special configuration of $A$, then, since $\OO_1$ is not part of a 1-1 configuration, also $\OO_2$ is not part of any 2-special configuration of $A$. Then all 2-special configurations of $A$ are contained in $\H_3$, hence one just need to set $A'=A$. If, instead, $\OO_1$ is part of a 2-special configuration $\s=\OO_1\OO_{2}\ldots\OO_r$ of $A$,  then, by assumption, $r>2$ and $n_2\equiv 2$ modulo 3. Hence, we have $n_2-1\equiv 1$ modulo 3 and $\s'=\OO_2'\OO_3\ldots\OO_r$ is a 2-special configuration of $\H_3$. Writing $A=\{\s,\s_1,\ldots,\s_t\}$, then $A'=\{\s',\s_1,\ldots,\s_t\}$ is a maximal set of disjoint 2-special configurations of $\H_3$, proving $M(\H)\leq M(\H_3)$. The other inequality follows similarly, giving $M(\H)=M(\H_3)$.
 
 It remains to show that if $\OO_1$ is part of a 1-1 special configuration, then $M(\H_3)=M(\H)-1$. By assumption, $n_2\equiv 1$ modulo 3, hence $n_2-1\equiv 0$ modulo 3. Then, by Remark \ref{2spec}, the string $\OO_2'$ is not part of any 2-special configuration of $\H_3$. Then, if $A'$ is a maximal set of disjoint  2-special configurations of $\H_3$, one has that $A=\{\s\}\cup A$ is a maximal set of disjoint 2-special configurations of $\H$, where $\s=\OO_1\OO_2$. This proves $M(\H_3)\leq M(\H)-1$.
 To see the other inequality, let $A=\{\s_0,\s_1,\ldots,\s_m\}$ be a maximal set of disjoint 2-special configurations of $\H$.
 If $\overline{\s}=\OO_1\OO_2$ form a 2-special configuration that is disjoint from every other 2-special configuration of $\H$, then, by Remark \ref{2spec}, we have $\overline{\s}$ is in $A$, say $\s_0=\overline{\s}$. The statement now follows because $A'=\{\s_1,\ldots,\s_m\}$ is a set of disjoint 2-special configurations of $\H_3$. We may then assume $\OO_2$ is also contained in another 2-special configuration $\widetilde{\s}=\OO_2\ldots\OO_r$ of $\H$, for some $r\geq 3$. 
 Since  every maximal set of disjoint 2-special configurations of $\H$ contains either $\overline{\s}$ or $\widetilde{\s}$, we may assume either $\s_0=\overline{\s}$ or $\s_0=\widetilde{\s}$. Then, similarly to the above, $A'=\{\s_1,\ldots,\s_m\}$ is a set of disjoint 2-special configurations contained in $\H_3$, showing that $M(\H_3)\geq M(\H)-1$.
\end{proof}

\begin{Lemma}\label{b}
Let $\H$ be a string with $\mu\geq 3$ vertices. Assume $\{2\}\notin \H$. Then $b(\H_3)=b(\H)-2$ if and only if $2$ is an isolated open vertex contained in a 1-1 configuration. In all other cases one has $b(\H_3)=b(\H_3)-1$.
\end{Lemma}

\begin{proof}
Let $\OO_1,\ldots,\OO_s$ be the strings of opens of $\H$ and let $n_1$ be the number of open vertices in the string $\OO_1$ containing the vertex $2$. If $n_1\geq 4$ then the strings of opens of $\H_3$ are $\OO_1',\OO_2,\ldots,\OO_s$, where $\OO_1'$ contains $n_1-3$ open vertices. Then, we have the equalities $s(\H)=s(\H_3)$ and
$$\sum_{i=1}^{s}\left\lfloor\frac{(n_i(\H_3)-1}{3}\right\rfloor= 
\left\lfloor\frac{(n_1(\H)-3)-1}{3}\right\rfloor+\sum_{i=2}^{s}\left\lfloor\frac{(n_i(\H)-1}{3}\right\rfloor=\left( \sum_{i=1}^{s}\left\lfloor\frac{(n_i(\H)-1}{3}\right\rfloor\right) -1$$
from which the statement follows. 
If $n_1\leq 3$, the statement follows from Table \ref{T1} below.
\end{proof}
\begin{table}
$$\begin{array}{|l||c|c|c|}\hline
& s(\H_3) & \sum\limits_{i=1}^{s(\H_3)}\left\lfloor\frac{n_i(\H_3)-1}{3}\right\rfloor & b(\H_3) \\\hline\hline

\mbox{if } n_1>3 & s(\H) & \left(\sum\limits_{i=1}^{s(\H)}\left\lfloor\frac{n_i(\H)-1}{3}\right\rfloor \right)-1 & b(\H)-1 \\\hline

\mbox{if either }2\leq n_1\leq 3, & & & \\
\mbox{or } n_1=1\;\mathrm{ and}\; n_2\not\equiv 1\; {\mathrm mod}\; 3, & s(\H)-1 & \sum\limits_{i=1}^{s(\H)}\left\lfloor\frac{n_i(\H)-1}{3}\right\rfloor & b(\H)-1\\
\mbox{or }n_1=1\mbox{ and }\{3\}\in \H\mbox{ and } \{4\}\in \H & & & \\\hline

\mbox{if }\n_1=1, n_2\equiv 1\; {\mathrm mod}\; 3 \mbox{ and }n_2>1  & s(\H)-1 & \left(\sum\limits_{i=1}^{s(\H)}\left\lfloor\frac{n_i(\H)-1}{3}\right\rfloor\right) -1 & b(\H)-2 \\\hline

\mbox{if }n_1=n_2=1 & s(\H)-2 & \sum\limits_{i=1}^{s(\H)}\left\lfloor\frac{n_i(\H)-1}{3}\right\rfloor & b(\H)-2 \\\hline
\end{array}$$
\caption{}\label{T1}
\end{table}
We are now ready to prove Theorem \ref{string}.\\

\noindent
{\bf Proof of Theorem~\ref{string}} 
\noindent
Recall that $b(\H)=s(\H)+\sum_{i=1}^{s(\H)}\left\lfloor\frac{n_i(\H)-1}{3}\right\rfloor$.
Let $Exp(\H)=\mu(\H)-b(\H)+M(\H)$ be the expected formula for the projective dimension of $\H$. We prove the equality $\pd(\H)=Exp(\H)$ by induction on the number of vertices $\mu=\mu(\H)$.\\
If $\mu\leq 2$, then $\H$ is saturated, hence the statement follows by Proposition \ref{saturated}. We may then assume $\mu\geq 3$ and the statement is proved for any hypergraph having at most $\mu-1$ vertices.
If the vertex $2$, which is the neighbor of the endpoint $1$, is closed, then $Exp(\H)=Exp(\H_1)+1$ and, by Lemma \ref{red}, we have $\pd(\H)=\pd(\H_1)+1$. Since the induction hypothesis gives $\pd(\H_1)=Exp(\H_1)$, we immediately obtain $\pd(\H)=Exp(\H)$. 

We may then assume the vertex $2$ is open. By Proposition \ref{red3}, we have $\pd(\H)=\pd(\H_3)+2$, and by induction hypothesis $Exp(\H_3)=\pd(\H_3)$. Hence, to prove the statement it suffices to show that $Exp(\H)=Exp(\H_3)+2$. Since $\mu(\H)=\mu(\H_3)+3$, this is equivalent to  prove $b(\H_3)-M(\H_3)=b(\H)-M(\H)-1.$
 This follows by Lemmas \ref{H3} and \ref{b}.
\QED

\section{Projective Dimension of Cycles}

In this section we prove an analogous formula for the projective dimension of cycles. A cycle is obtained by identifying the two endpoints of a string.
\begin{Definition}\label{Defcycle}
Fix $\mu\geq 3$. A hypergraph $\H$ is a $\mu$-{\em cycle} if $\H$ can be written as $\H=\widetilde{\H}\cup \{\mu,1\}$, where $\widetilde{\H}$ is a string over $V=[\mu]$.
\end{Definition} 
When the number of vertices is not relevant, we omit it and just say that $\H$ is a {\it cycle}.  Examples of cycles are $\H'=\{\{1,2\}, \{2,3\}, \{3,4\}, \{4,1\}\}$, and $\H''=\{\{1\}, \{1,2\}, \{2,3\}, \{3\}, \{3,4\}, \{4,1\}\}$.
In contrast with the string case, a cycle may have only open vertices, for instance $\H'$ above, or exactly one closed vertex. We will then need to isolate these two situations when defining a `string of opens' inside a cycle.
Write $\mu=\mu(\H)$.  
\begin{itemize}
\item If $\H$ contains at least two closed vertices, we set $s=s(\H)$ to be the number of strings of opens in $\H$ and $n_1(\H),\ldots,n_{s}(\H)$ to be the number of opens in each string of open;
\item if $\H$ contains at most one closed vertex, we set $s=s(\H)=1$ and $n_1(\H)=\mu-1$;
\item we set $b(\H)=s(\H)+\sum_{i=1}^{s(\H)}\left\lfloor\frac{n_i(\H)-1}{3}\right\rfloor$, in analogy with the string case. 
\end{itemize}
Finally, the definition of a {\em 2-special configuration} $\mathcal{S}$ in $\H$ is the same as the definition of a 2-special configuration in a string, with the exception that one allows the two extremal vertices of $\mathcal{S}$ to coincide (if this happens then the entire cycle is itself a 2-special configuration, the smallest example of which is $\H=\{\{1\}, \{1,2\}, \{2,3\}, \{3\}, \{3,4\}, \{4,1\} \}$). Two 2-special configurations in a cycle $\H$ are {\em disjoint} if they do not share any open vertex. The {\em modularity} $M(\H)$ of a cycle $\H$ is the maximal number of pairwise disjoint 2-special configurations in $\H$.

\begin{Example}\label{C1}
Let $\mathcal{H}=\{\{1\},\{1,2\},\{2,3\},\{3\},\{3,4\},\{4,5\},\{5,6\},\{6,7\},\{7,8\},\{8\},\{8,9\},$
$\{9\},\{9,10\},\{10\},\{10,11\},\{11,12\},\{12\},\{12,13\},\{13,14\},\{14,15\},\{15\},\{15,16\},\{16,1\}\}$, then $\mu(\H)=16$, $b(\H)=5+\left\lfloor\frac{4-1}{3}\right\rfloor=6$ and $M(\H)=2$. Figure \ref{cycle16} shows the hypergraph $\H$.
\begin{figure}[h] 

\begin{tikzpicture}[thick, scale=0.7]

\shade [shading=ball, ball color=black] (1,0) circle (.1) node [left] {$\mathcal{H}:$  } ; 
\draw  [shape=circle] (2,0) circle (.1) ;
\shade [shading=ball, ball color=black] (3,0) circle (.1) ;
\draw  [shape=circle] (4,0) circle (.1) ;
\draw  [shape=circle] (5,0) circle (.1) ;
\draw  [shape=circle] (6,0) circle (.1) ;
\draw  [shape=circle] (7,0) circle (.1) ;
\shade [shading=ball, ball color=black] (8,0) circle (.1) ;
\shade [shading=ball, ball color=black] (8,-1) circle (.1) ;
\shade [shading=ball, ball color=black] (7,-1) circle (.1) ;
\draw  [shape=circle] (6,-1) circle (.1) ;
\shade [shading=ball, ball color=black] (5,-1) circle (.1) ;
\draw  [shape=circle] (4,-1) circle (.1) ;
\draw  [shape=circle] (3,-1) circle (.1) ;
\shade [shading=ball, ball color=black] (2,-1) circle (.1) ;
\draw  [shape=circle] (1,-1) circle (.1) ;

\draw [line width=1.2pt  ] (1.1,0)--(1.9,0);
\draw [line width=1.2pt  ] (2.1,0)--(2.9,0);
\draw [line width=1.2pt  ] (3.1,0)--(3.9,0);
\draw [line width=1.2pt  ] (4.1,0)--(4.9,0);
\draw [line width=1.2pt  ] (5.1,0)--(5.9,0);
\draw [line width=1.2pt  ] (6.1,0)--(6.9,0);
\draw [line width=1.2pt  ] (7.1,0)--(7.9,0);
\draw [line width=1.2pt  ] (1.1,-1)--(1.9,-1);
\draw [line width=1.2pt  ] (2.1,-1)--(2.9,-1);
\draw [line width=1.2pt  ] (3.1,-1)--(3.9,-1);
\draw [line width=1.2pt  ] (4.1,-1)--(4.9,-1);
\draw [line width=1.2pt  ] (5.1,-1)--(5.9,-1);
\draw [line width=1.2pt  ] (6.1,-1)--(6.9,-1);
\draw [line width=1.2pt  ] (7.1,-1)--(7.9,-1);
\draw [line width=1.2pt  ] (1,-0.1)--(1,-0.9);
\draw [line width=1.2pt  ] (8,-0.1)--(8,-0.9);

\end{tikzpicture}

\caption{}\label{cycle16}
\end{figure}
\end{Example}

We now state the main result of this section, providing a closed formula for the projective dimension of any cycle. Remarkably, it is the same formula found for strings (see Theorem \ref{string}). 
\begin{Theorem}\label{cycles2}
If $\H$ is a cycle hypergraph, then $\pd(\H)=\mu(\H)-b(\H)+M(\H).$
\end{Theorem}
For instance, the cycle $\H$ of Example \ref{C1} has $\pd(\H)=16-6+2=12$.
For a cycle $\H$ we denote the expected formula for the projective dimension by $Exp(\H)=\mu(\H)-b(\H)+M(\H)$. Hence, to prove Theorem \ref{cycles2}, we need to show that $\pd(\H)=Exp(\H)$ for every cycle.

 In \cite{J} it was proved that, if a cycle $\H$ only contains open vertices, then $\pd(\H)=\lfloor\frac{\mu}{3}\rfloor+\lceil\frac{\mu}{3}\rceil$.
Since one has $\lfloor\frac{\mu}{3}\rfloor+\lceil\frac{\mu}{3}\rceil=\mu-1-\lfloor\frac{\mu-2}{3}\rfloor=Exp(\H)$, we obtain the following.

\begin{Proposition}\label{open}
If $\H$ is a cycle having only open vertices, then $\pd(\H)=Exp(\H)$.
\end{Proposition}

Next, we prove the formula for small cycles.
\begin{Lemma}\label{3or4}
If $\H$ is a $\mu$-cycle hypergraph with $\mu\leq 4$, then $\pd(\H)=Exp(\H)$.
\end{Lemma}
\begin{proof}
If $\H$ is saturated, then $Exp(\H)=\mu$ and, by Proposition \ref{saturated}, also $\pd(\H)=\mu$.

We may then assume $\H$ is not saturated.
Let $\H_0$ be the $\mu$-cycle hypergraph whose vertices are all open, and let $\overline{\H}$ be the $\mu$-cycle hypergraph whose vertices are all closed. By Proposition \ref{open} we have $\pd(\H_0)=\mu-1$, and by iterated use of Corollary \ref{compare} we haves $\pd(\H_0)\leq \pd(\H)$. Finally, Proposition \ref{saturated} gives $\pd(\H)<\mu=\pd(\overline{\H})$, yielding
$$\mu-1=\pd(\H_0) \leq \pd(\H)<\mu.$$
Then we have $\pd(\H)=\mu-1$, and it is easily checked that also $Exp(\H)=\mu-1$.
\end{proof}

By Proposition \ref{open} and Lemma \ref{3or4} we only need to prove Theorem \ref{cycles2} for cycles containing at least $5$ vertices, at least one of which is closed.

Next, we prove the formula when $\H$ contains at least two adjacent closed vertices.
\begin{Lemma}\label{cons}
Let $\H$ be a cycle. If $\H$ contains two adjacent closed vertices, then $\pd(\H)=Exp(\H)$.
\end{Lemma}
\begin{proof}
Let $i$ and $i+1$ be two adjacent closed vertices of $\H$, and let $E_i=\{i,i+1\}$ be the face of $\H$ containing both of them. Then $\H'=\H\setminus E_i$ is a string, and clearly $\mu(\H')=\mu(\H)$. Moreover, since we removed from $\H$ an edge connecting two closed vertices, we have $b(\H)=b(\H')$ and $M(\H)=M(\H')$. These equalities, together with Theorem \ref{string}, give
$$\pd(\H')=\mu(\H')-b(\H')+M(\H')=\mu(\H)-b(\H)+M(\H)=Exp(\H).$$
Finally, by Proposition \ref{eq}, we have $\pd(\H)=\pd(\H')$, which concludes the proof.
\end{proof}

 The following result essentially reduces the remaining portion of the problem to the string case.
For any collection of vertices $V'=\{v_1,\ldots,v_r\}$ of the hypergraph $\H$, let $\H_{V'}$ be the hypergraph whose faces are obtained from the faces of $\H$ as follows: for any face $F$ of $\H$, if $F$ does not contain any vertex of $V'$, then $F$ is also a face of $H_{V'}$; if $F$ contains the vertices $v_{j_1},\ldots,v_{j_s}$ of $V'$, then the face $F'=F\setminus\{v_{j_1},\ldots,v_{j_s}\}$ is a face of $\H_{V'}$. If $I=(m_i\,|\,i\in [\mu])$ is an ideal associated to $\H$, then $I'=(m_i\,|\,i\in [\mu]\setminus V')$ is an ideal associated to $H_{V'}$.
\begin{Lemma}\label{cycles1}
Let $\H$ be a $\mu$-cycle hypergraph for some $\mu\geq 5$. Assume $\{v_1\}\in \H$ for some $v_1$.  Let $v_2$ and $v_3$ be the neighbors of $v_1$, and $v_4$ and $v_5$ be the other neighbor of $v_2$ and $v_3$, respectively. Set $\s_1=H_{\{v_1\}}$ and $\s_5=\H_{\{v_1,v_2,v_3,v_4,v_5\}}$. Then $\pd(\H)={\rm max}\{\pd(\s_1),\pd(\s_5)+3\}$.
\end{Lemma}

\begin{proof}
We set $\Q_1=\s_1:v_1$. By Lemma \ref{ref}, we have $\pd(\H)={\rm max}\{\pd(\s_1),\pd(\Q_1)+1 \}.$
Since $v_2$ and $v_3$ are open vertices, then $\Q_1$ is the disjoint union of $\s_5$ and two vertices, whence $\pd(\Q_1)=\pd(\s_5)+2$. Then, $\pd(\H)={\rm max}\{\pd(\s_1),\pd(\s_5)+3 \}.$
\end{proof}

 One may hope that, in Lemma \ref{cycles1}, either $\pd(\s_1)$ or $\pd(\s_5)+3$ is always larger than the other number, so that one could either have $\pd(\H)=\pd(\s_1)$ or $\pd(\H)=\pd(\s_5)+3$. The following examples show that this is not possible, in general.

\begin{Example}\label{max}
Let $\H$ be the hypergraph $H=\{ \{1\}, \{1,2\}, \{2,3\}, \{3\}, \{3,4\}, \{4,5\}, \{5,6\}, \{6,7\},$
$ \{7\}, \{7,8\}, \{8,1\}\}$, and let $v_1$ be the vertex $1$. Then one has $\pd(\s_1)=6$ and $\pd(\s_5)=2$, giving $${\rm max}\{\pd(\s_1),\pd(\s_5)+3\}=\pd(\s_1)>\pd(\s_5)+3.$$

\noindent Now, let $\H'$ be the hypergraph $\H'=\{\{1\}, \{1,2\}, \{2,3\}, \{3,4\}, \{4,5\}, \{5,6\}, \{6,7\}, \{7,8\}, \{8,1\}\}$, and let $v_1$ be the vertex $1$.  Then one has $${\rm max}\{\pd(\s'_1),\pd(\s'_5)+3\}=6=\pd(\s'_5)+3>\pd(\s'_1)=5.$$

The hypergraphs of Example \ref{max} are shown in Figure \ref{C1C5}. 
\begin{figure}[h] 

\begin{tikzpicture}[thick, scale=0.7]

\draw  [line width=0.000001pt] (1.4,0.55)--(1.40000001,0.55) node [left] {$v_1$ }  ;
\shade [shading=ball, ball color=black] (1,0) circle (.1) node [left] {$\mathcal{H}:$ }  ; 
\draw  [shape=circle] (2,0) circle (.1) ;
\shade [shading=ball, ball color=black] (3,0) circle (.1) ;
\draw  [shape=circle] (4,0) circle (.1) ;
\draw  [shape=circle] (4,-1) circle (.1) ;
\draw  [shape=circle] (3,-1) circle (.1) ;
\shade [shading=ball, ball color=black] (2,-1) circle (.1) ;
\draw  [shape=circle] (1,-1) circle (.1) ;

\draw [line width=1.2pt  ] (1.1,0)--(1.9,0);
\draw [line width=1.2pt  ] (2.1,0)--(2.9,0);
\draw [line width=1.2pt  ] (3.1,0)--(3.9,0);
\draw [line width=1.2pt  ] (1.1,-1)--(1.9,-1);
\draw [line width=1.2pt  ] (2.1,-1)--(2.9,-1);
\draw [line width=1.2pt  ] (3.1,-1)--(3.9,-1);
\draw [line width=1.2pt  ] (1,-0.1)--(1,-0.9);
\draw [line width=1.2pt  ] (4,-0.1)--(4,-0.9);

\draw  [line width=0.000001pt] (6.4,0.55)--(6.40000001,0.55) node [left] {$v_1$ }  ;
\shade [shading=ball, ball color=black] (6,0) circle (.1) node [left] {$\mathcal{H}':$  } ; 
\draw  [shape=circle] (7,0) circle (.1) ;
\draw  [shape=circle] (8,0) circle (.1) ;
\draw  [shape=circle] (9,0) circle (.1) ;
\shade [shading=ball, ball color=black] (9,-1) circle (.1) ;
\draw  [shape=circle] (8,-1) circle (.1) ;
\draw  [shape=circle] (7,-1) circle (.1) ;
\draw  [shape=circle] (6,-1) circle (.1) ;

\draw [line width=1.2pt  ] (6.1,0)--(6.9,0);
\draw [line width=1.2pt  ] (7.1,0)--(7.9,0);
\draw [line width=1.2pt  ] (8.1,0)--(8.9,0);
\draw [line width=1.2pt  ] (6.1,-1)--(6.9,-1);
\draw [line width=1.2pt  ] (7.1,-1)--(7.9,-1);
\draw [line width=1.2pt  ] (8.1,-1)--(8.9,-1);
\draw [line width=1.2pt  ] (6,-0.1)--(6,-0.9);
\draw [line width=1.2pt  ] (9,-0.1)--(9,-0.9);
\end{tikzpicture}

\caption{}\label{C1C5}
\end{figure}
\end{Example}
In contrast, the following Lemma -- that will be used in the proof of Proposition \ref{reduc} -- shows that if a closed vertex delimits a string of precisely 3 opens, then $\pd(\H)=\pd(\s_5)+3$. 
\begin{Lemma}\label{st}
In the setting of Lemma \ref{cycles1} further assume $\mu(\H)\geq 7$, and the string of opens containing $v_3$ has exactly 3 open vertices. Then $\pd(\H)=\pd(\s_5)+3$.
\end{Lemma}
\begin{proof}
Since $\s_1$ is obtained by adding $4$ vertices to $\s_5$, iterated use of Lemma \ref{red2} gives $\pd(\s_1)\leq\pd(\s_3)+2\leq \pd(\s_5)+4$, where $\s_3=\H_{\{v_1, v_2, v_4\}}$. Also, by Lemma \ref{cycles1} we have $\pd(\H)={\rm max}\{\pd(\s_1),\pd(\s_5)+3 \}$, hence it is enough to show that $\pd(\s_1)\,<\, \pd(\s_5)+4$. By the above chain of inequalities, it suffices to prove $\pd(\s_3)+2\neq \pd(\s_5)+4$, that is, $\pd(\s_3)\neq \pd(\s_5)+2$.

Let $v_6$ be the other neighbor in $\H$ of $v_5$, and $v_7$ the other neighbor in $\H$ of $v_6$. By assumption, $v_6$ is open in $\H$ and  $v_7$ is closed. Note that $v_6$ is an endpoint of $\s_5$, hence $v_6$ is closed in $\s_5$, and its only neighbor in $\s_5$ is $v_7$, which is again closed. Then, by Lemma \ref{red}, we have $\pd(\s_5)=\pd(\s_6)+1$, where $\s_6$ is obtained by removing from $\s_5$ the faces containing $v_6$.
Since $\s_3$ is obtained by appending to one end of $\s_6$ a string of 3 adjacent vertices whose second vertex is open, by Proposition \ref{red3} we have $\pd(\s_3)=\pd(\s_6)+2$. Then $\pd(\s_3)=\pd(\s_5)+1\neq \pd(\s_5)+2$.
\end{proof}

The next result captures another interesting behavior of projective dimension, similar to the one of Proposition \ref{red3}. Roughly speaking, it states that, whenever we add three adjacent open vertices to a cycle hypergraph, the projective dimension increases exactly by 2 units.  
\begin{Proposition}\label{reduc}
Let $\H'$ be a cycle hypergraph, and let $\H$ be the hypergraph obtained by adding $3$ connected open vertices to $\H'$, then $\pd(\H)=\pd(\H')+2$.
\end{Proposition}

\begin{proof}
If all the vertices of $\H$ are open, then also all the vertices of $\H'$ are open, and the statement follows by Proposition \ref{open}. Also, if $\H$ contains two adjacent closed vertices, then so does $\H'$, and the statement follows by Lemma \ref{cons}. We may then assume $\H$ contains at least one closed vertex, but $\H$ does not contain two adjacent closed vertices. If $\mu(\H')=3$, then by the above, we may assume either $\H'$ contains exactly two closed vertices and the three connected open are placed in between them, or $\H'$ contains exactly one closed vertex. In the former case, the cycle $\H$ is isomorphic to $\widehat{\H}=\{\{1\}, \{1,2\}, \{2,3\}, \{3,4\}, \{4,5\}, \{5\}, \{5,6\}, \{6,1\}  \}$, whereas, in the former case, $\H$ is isomorphic to $\widehat{\H}=\{ \{1\}, \{1,2\}, \{2,3\}, \{3,4\}, \{4,5\}, \{5,6\}, \{6,1\} \}$. In either case, Lemma \ref{cycles1} gives $\pd(\H)=4$, and the statement follows because, by Lemma \ref{3or4}, one has $\pd(\H')=2$. The case where $\mu(\H')=4$ can be proved similarly.

We may then assume $\mu(\H')\geq 5$, so that $\mu(\H)\geq 8$.  Set $s=s(\H')$. Without loss of generality we may assume one of the three open vertices has a closed vertex $v_1$ as a neighbor in $\H$. Let $v_2$ and $v_3$ be the neighbors of $v_1$ in $\H$ and $v_4$ and $v_5$ be the other neighbor in $\H$ of $v_2$ and $v_3$, respectively. We may assume $v_3$, $v_5$ and the other neighbor, $v_6$, of $v_5$ are the three connected opens that have been added to $\H'$. Since $\H$ does not contain two adjacent closed vertices, we may also assume that the vertex $v_2$ is open. Now, setting $\s_1=\H_{\{v_1\}}$ and $s_5=H_{\{v_1,\ldots,v_5\}}$, we have $$\pd(\H)={\rm max}\{\pd(\s_1),\pd(s_5)+3\}.$$

Now, let $v_3'$ be the other neighbor of $v_6$ and $v_5'$ be the other neighbor of $v_3'$ in $\H$. Since $v_3$, $v_5$ and $v_6$ are the three connected open vertices added to $\H'$, then $v_3'$ is a neighbor of $v_1$ in $\H'$.
Set $\s_1'=H'_{\{v_1\}}$ and $\s_5'=\H'_{\{v_1,v_2,v_3',v_4,v_5'\}}$.

If $v_3'$ is open in $\H'$ then, by Lemma \ref{cycles1}, we have  $\pd(\H')={\rm max}\{\pd(\s'_1),\pd(\s'_5)+3\}$. Since Theorem \ref{string} yields the equalities  $\pd(\s_1)=\pd(\s'_1)+2$ and $\pd(\s_5)=\pd(\s'_5)+2$, we have
$$\pd(\H)= {\rm max}\{\pd(\s_1),\pd(\s_5)+3\}={\rm max}\{\pd(\s'_1)+2,\pd(\s'_5)+2+3\}=$$
$$={\rm max}\{\pd(\s'_1),\pd(\s'_5)+3\}+2=\pd(\H')+2.$$

We may then assume $v_3'$ is closed in $\H'$. By Lemma \ref{cons} we have $\pd(\H')=\pd(\s')$ where $\s'$ is the string $\s'=\H'\setminus E$, where $E$ denotes the face $E=\{v_1,v_3'\}$. Since $v_2$ is open in $\s'$ and is the neighbor of the endpoint $v_1$ of $\s'$, then, by Proposition \ref{red3}, $\pd(\s')=\pd(\s'_3)+2$ where $\s'_3$ is obtained by removing the vertices $v_1, v_2$ and $v_4$ from $\s'$. 
Since $\s_5$ can be obtained by appending a closed vertex to the endpoint $v_3'$ of $\s'_3$, Lemma \ref{red} gives  $\pd(\s_5)=\pd(\s'_3)+1$. Thus, by all the above, we have $$\pd(\s_5)+3=\pd(\s'_3)+4=\pd(\s')+2=\pd(\s')+2.$$ 
Hence, it suffices to show that, in this setting, $\pd(\H)=\pd(\s_5)+3$. This was proved in Lemma \ref{st}.
\end{proof}

Proposition \ref{reduc} allows us to reduce the size of the cycle hypergraph to compute its projective dimension. We illustrate this in Example \ref{3}, which is shown in Figure \ref{re3o}. 
\begin{Example}\label{3}
Let $\H=\{\{1\}, \{1,2\}, \{2,3\}, \{3,4\}, \{4,5\}, \{5\}, \{5,6\}, \{6,7\}, \{7,8\}, \{8,9\}, \{9,10\},$
 $\{10,1\} \}$, $\H'=\{\{1\}, \{1,2\}, \{2,3\}, \{3,4\}, \{4,5\}, \{5\}, \{5,6\}, \{6,7\},\{7,1\} \}$ and $\H''=\{ \{1\},$
 $\{1,2\}, \{2\}, \{2,3\}, \{3,4\}, \{4,1\} \}$.

Then $\pd(\H)=\pd(\H')+2=\pd(\H'')+4$. By Lemma \ref{3or4} we have $\pd(\H'')=3$, so that $\pd(\H)=3+4=7$.
\begin{figure}[h] 

\begin{tikzpicture}[thick, scale=0.7]
\shade [shading=ball, ball color=black] (1,0) circle (.1) node [left] {$\mathcal{H}:$  } ; 
\draw  [shape=circle] (2,0) circle (.1) ;
\draw  [shape=circle] (3,0) circle (.1) ;
\draw  [shape=circle] (4,0) circle (.1) ;
\shade [shading=ball, ball color=black] (5,0) circle (.1) ;
\draw  [shape=circle] (1,-1) circle (.1) ;
\draw  [shape=circle] (2,-1) circle (.1) ;
\draw  [shape=circle] (3,-1) circle (.1) ;
\draw  [shape=circle] (4,-1) circle (.1) ;
\draw  [shape=circle] (5,-1) circle (.1) ;
\draw [line width=1.2pt  ] (1.1,0)--(1.9,0);
\draw [line width=1.2pt  ] (2.1,0)--(2.9,0);
\draw [line width=1.2pt  ] (3.1,0)--(3.9,0);
\draw [line width=1.2pt  ] (4.1,0)--(4.9,0);
\draw [line width=1.2pt  ] (1.1,-1)--(1.9,-1);
\draw [line width=1.2pt  ] (2.1,-1)--(2.9,-1);
\draw [line width=1.2pt  ] (3.1,-1)--(3.9,-1);
\draw [line width=1.2pt  ] (4.1,-1)--(4.9,-1);
\draw [line width=1.2pt  ] (1,-0.1)--(1,-0.9);
\draw [line width=1.2pt  ] (5,-0.1)--(5,-0.9);

\shade [shading=ball, ball color=black] (7,0) circle (.1) node [left] {$\mathcal{H'}:$  } ; 
\draw  [shape=circle] (8,0) circle (.1) ;
\draw  [shape=circle] (9,0) circle (.1) ;
\draw  [shape=circle] (10,0) circle (.1) ;
\shade [shading=ball, ball color=black] (11,0) circle (.1) ;
\draw  [shape=circle] (7,-1) circle (.1) ;

\draw  [shape=circle] (11,-1) circle (.1) ;
\draw [line width=1.2pt  ] (7.1,0)--(7.9,0);
\draw [line width=1.2pt  ] (8.1,0)--(8.9,0);
\draw [line width=1.2pt  ] (9.1,0)--(9.9,0);
\draw [line width=1.2pt  ] (10.1,0)--(10.9,0);
\draw [line width=1.2pt  ] (7.1,-1)--(10.9,-1);
\draw [line width=1.2pt  ] (7,-.1)--(7,-.9);
\draw [line width=1.2pt  ] (11,-.1)--(11,-.9);

\shade [shading=ball, ball color=black] (13,0) circle (.1) node [left] {$\mathcal{H''}:$  } ; 
\shade [shading=ball, ball color=black] (17,0) circle (.1) ;
\draw  [shape=circle] (13,-1) circle (.1) ;
\draw  [shape=circle] (17,-1) circle (.1) ;
\draw [line width=1.2pt  ] (13.1,0)--(16.9,0);
\draw [line width=1.2pt  ] (13.1,-1)--(16.9,-1);
\draw [line width=1.2pt  ] (13,-.1)--(13,-.9);
\draw [line width=1.2pt  ] (17,-.1)--(17,-.9);

\end{tikzpicture}

\caption{}\label{re3o}
\end{figure}
\end{Example}
\begin{Remark}\label{exp}
Let $\H'$ be a cycle hypergraph, and let $\H$ be the hypergraph obtained by adding $3$ connected open vertices to $\H'$, then $Exp(\H)=Exp(\H')+2$.
\end{Remark}
Remark \ref{exp} and Proposition \ref{reduc} imply that $\pd(\H)=Exp(\H)$ if and only if $\pd(\H')=Exp(\H')$. Hence, in the proof of Theorem \ref{cycles2}, we may assume all the strings of opens contain at most two open vertices. 

The following remark follows immediately from the fact that every special configuration begins and ends with a string of opens having a number of open vertices that is congruent to $1$ modulo 3.
\begin{Remark}\label{modrem}
Let $\H$ be a cycle hypergraph. Assume $\H$ has exactly $r\leq s(\H)$ strings of open vertices whose number of open vertices is congruent to $1$ modulo $3$. Then $M(\H)\leq \lfloor \frac{r}{2}\rfloor$.
\end{Remark}

Recall that an open vertex of a hypergraph is {\it isolated} if all of its neighbors are closed.
\begin{Lemma}\label{modul}
Assume either 
\begin{itemize}
\item $\H$ is a string not containing 2 adjacent closed vertices except, possibly, at the endpoints, or
\item $\H$ is a cycle not containing 2 adjacent closed vertices.
\end{itemize}
If every string of opens in $\H$ contains at most two open vertices, then $M(\H)=\lfloor\frac{Is(\H)}{2}\rfloor$, where $Is(\H)$ denotes the number of isolated open vertices in $\H$.
\end{Lemma}

\begin{proof}
Since the proofs in the two cases are similar, we only prove the statement for a cycle $\H$. Let $t=\lfloor\frac{Is(\H)}{2}\rfloor$ be the quotient of the division of $Is(\H)$ by 2. Note that, by assumption, if $\H$ contains a special configuration, then it must contain an isolated open vertex. Hence, if $\H$ contains no isolated open vertices, then $M(\H)=0$, and the statement follows. We may then assume $\H$ contains at least one isolated open vertex. Let $v_1,v_2,\ldots,v_{Is(\H)}$ be the isolated open vertices of $\H$, labelled in clockwise order starting from $v_1$. For every $i=1,\ldots,t$, let $\mathcal{A}_{i}$ be the string consisting of all vertices of $\H$ between the closed vertex preceding  $v_{2i-1}$ and the closed vertex following $v_{2i}$ (in clockwise order). 
Then,  the strings $\mathcal{A}_1,\ldots,\mathcal{A}_t$ are 2-special configurations in $\H$, and by construction, are all disjoint, showing that $M(\H)\leq t$.  On the other hand, by Remark \ref{modrem}, we have $M(\H)\leq t$, whence $M(I)=t=\lfloor\frac{Is(\H)}{2}\rfloor$.
\end{proof}

We can now finish the proof of Theorem \ref{cycles2}.\\

\noindent
{\bf Proof of Theorem~\ref{cycles2}.} 
\noindent
 It suffices to show $\pd(\H)=Exp(\H)$. By Lemma \ref{cons} we may assume $\H$ does not contain two adjacent closed vertices. Also, by Proposition \ref{reduc} and Remark \ref{exp}, we may assume every string of open in $\H$ contains at most two open vertices, and, by Proposition \ref{open}, we may assume there is at least one closed vertex in $\H$. In this setting, if $\H$ contains only one closed vertex, then $\H$ has at most (hence, exactly) $3$ vertices, and the statement follows by Lemma \ref{3or4}. We may then assume there are at least two adjacent strings of opens, say $\s_1$ and $\s_2$ having $n_1\geq 1$ and $n_2\geq 1$ open vertices, respectively. Note that $2\leq n_1+n_2\leq 4$, and we can choose the strings of opens so that $n_1+n_2$ is maximal. Let $v_1$ be the closed vertex separating $\s_1$ and $\s_2$, by Lemma \ref{cycles1} we have $\pd(\H)={\rm max}\{\pd(\s_1),\pd(\s_5)+3\}$. 

If $n_1=n_2=2$, we have $\mu(\s_1)=\mu(\H)-1$, $s(\s_1)=s(\H)$ and, by Lemma \ref{modul},  $M(\s_1)=M(\H)+1$, giving $\pd(\s_1)=\mu(\H)-s(\H)+M(\H)=Exp(\H)$.
Analogously, we have $\mu(\s_5)=\mu(\H)-5$, $s(\s_5)=s(\H)-2$ and $M(\s_5)=M(\H)$, showing that $\pd(\s_5)+3=Exp(\H)=\pd(\s_1)$, which implies $\pd(\H)=Exp(\H).$

We may then assume $n_1=1$ and $n_2\leq 2$. Note that if $n_2=2$ and $\H$ has exactly two strings of open, then $\H$ is a pentagon with exactly two closed vertices, which are non adjacent. Lemma \ref{cycles1} then yields $\pd(\H)=3=Exp(\H)$. We may then assume $\H$ contains $s\geq 3$ distinct strings of open vertices, and let $1\leq n_3\leq 2$ be the number of open vertices in the strings of open near $\s_1$. 
We have $s(\s_1)=s(\H)-1$ and $M(\s_1)=M(\H)$, so that $\pd(\s_1)=\mu(\H)-s(\H)+M(\H)=Exp(\H)$. Similarly to the above, elementary computations combined with Lemma \ref{modul} prove that $$s(\s_5)=\left\{\begin{array}{ll}
s(\s)-2,& \mbox{ if }n_3=2\\
s(\s)-3,& \mbox{ if }n_3=1
 \end{array}\right.\;\mbox{ and }\;
 M(\s_5)=\left\{\begin{array}{ll}
  M(\H)-1,& \mbox{ if }n_3=2\\
 M(\H),& \mbox{ if } n_3=1
  \end{array}\right.$$
In any case one obtains $\pd(\s_5)+3=\mu(\H)-s(\H)+M(\H)=Exp(\H)$, and the formula follows. 

We may then assume every open in $\H$ is isolated. Then, $s(\s_1)=s(\H)-2$, $M(\s_1)=M(\H)-1$, $s(\s_5)=s(\H)-4$ and $M(\s_5)=M(\H)-2$, and the formula follows again.
\QED
\bigskip

Similarly to the string case, there is a number of cycles whose corresponding clutter is not edgewise dominant (see \cite{DS}), hence Theorem \ref{cycles2} is not covered by the main result of \cite{DS}. See the following simple example.

\begin{Example}\label{DS}
Let $\mathcal{H}$ and $\mathcal{C}$ be the hypergraph and clutter of $I=(ab,bcd,de,efg)$, see Figure~\ref{Ccycle}. Then, by Theorem \ref{string} we have $\pd(\mathcal{H})=4-2+1=3$, hence $\pd(\H)>\,|V(\mathcal{C})|-i(\mathcal{C})=6-4$. 
\begin{figure}[h] 
\begin{tikzpicture}[thick, scale=0.7]
\shade [shading=ball, ball color=black] (0,0) circle (.1) node [left] {$\mathcal{H}:$  } ; 
\draw  [shape=circle] (1,0) circle (.1) ;
\shade [shading=ball, ball color=black] (1,-1) circle (.1) ;
\draw  [shape=circle] (0,-1) circle (.1) ;
\draw [line width=1.2pt  ] (0,0)--(0.9,0);
\draw [line width=1.2pt  ] (1,-0.1)--(1,-0.9)  ;
\draw [line width=1.2pt  ] (0.1,-1)--(0.9,-1)  ;
\draw [line width=1.2pt  ] (0,-0.1)--(0,-0.9)  ;
\shade [shading=ball, ball color=black] (4,0) circle (.1);
\shade [shading=ball, ball color=black] (5,0) circle (.1) ;

\shade [shading=ball, ball color=black] (3,-.5) circle (.1)  node [left] {$\mathcal{C}:$  } ; 
\shade [shading=ball, ball color=black] (4,-1) circle (.1) ;
\shade [shading=ball, ball color=black] (5,-1) circle (.1) ;
\shade [shading=ball, ball color=black] (6,-.5) circle (.1) ;
\path [pattern=north east lines]   (4,0)--(4,-1)--(3,-.5)--cycle;
\path [pattern=north east lines]   (5,0)--(5,-1)--(6,-.5)--cycle;
\draw [line width=1.2pt  ] (4,0)--(4.9,0);
\draw [line width=1.2pt  ] (4,-1)--(4.9,-1)  ;
\end{tikzpicture}
\caption{}\label{Ccycle}
\end{figure}
\end{Example}

We now provide an example showing the easiness of applying the formula proved in Theorem \ref{cycles2} even to hypergraphs with a fairly high number of generators.

\begin{Example}\label{14}
Let $\mathcal{H}$ be the cycle shown in Figure~\ref{cycle}, then we have $\mu(\mathcal{H})=14, s(\mathcal{H})=4$ and $M(\mathcal{H})=1$, so that $\pd(\mathcal{H})=14-4-1+1=10$.
\begin{figure}[h] 
\begin{tikzpicture}[thick, scale=0.7]
\shade [shading=ball, ball color=black] (0,0) circle (.1) ;
\draw  [shape=circle] (1,0) circle (.1) ;
\shade [shading=ball, ball color=black] (2,0) circle (.1) ;
\draw  [shape=circle] (3,0) circle (.1) ;
\draw  [shape=circle] (4,0) circle (.1) ;
\shade [shading=ball, ball color=black] (5,0) circle (.1) ;
\draw  [shape=circle] (5.5,-0.5) circle (.1) ;
\shade [shading=ball, ball color=black] (5,-1) circle (.1) ;
\draw  [shape=circle] (4,-1) circle (.1) ;
\draw  [shape=circle] (3,-1) circle (.1) ;
\draw  [shape=circle] (2,-1) circle (.1) ;
\draw  [shape=circle] (1,-1) circle (.1) ;
\draw  [shape=circle] (0,-1) circle (.1) ;
\shade [shading=ball, ball color=black] (-0.5,-0.5) circle (.1) node [left] {$\mathcal{H}:$  } ; 
\draw [line width=1.2pt  ] (0,0)--(0.9,0);
\draw [line width=1.2pt  ] (1.1,0)--(1.9,0)  ;
\draw [line width=1.2pt  ] (2.1,0)--(2.9,0)  ;
\draw [line width=1.2pt  ] (3.1,0)--(3.9,0)  ;
\draw [line width=1.2pt  ] (4.1,0)--(4.9,0)  ;
\draw [line width=1.2pt  ] (5.05,-0.05)--(5.45,-0.45)  ;
\draw [line width=1.2pt  ] (5.45,-0.55)--(5.05,-.95)  ;
\draw [line width=1.2pt  ] (4.1,-1)--(4.9,-1)  ;
\draw [line width=1.2pt  ] (3.1,-1)--(3.9,-1)  ;
\draw [line width=1.2pt  ] (2.1,-1)--(2.9,-1)  ;
\draw [line width=1.2pt  ] (1.1,-1)--(1.9,-1)  ;
\draw [line width=1.2pt  ] (0.1,-1)--(0.9,-1)  ;
\draw [line width=1.2pt  ] (-0.05,-.05)--(-0.5,-0.45)  ;
\draw [line width=1.2pt  ] (-0.5,-0.55)--(-0.1,-0.95)  ;

\end{tikzpicture}

\caption{}\label{cycle}
\end{figure}
\end{Example}
We conclude this section with a characterization of the strings and cycles whose corresponding ideals are Cohen-Macaulay. Recall that an $R$-ideal $I$ is Cohen-Macaulay if and only if ${\rm grade}(I)=\pd(I)$. The following remark is an immediate consequence of \cite[Proposition~3.3]{KTY2}. 
\begin{Remark}\label{ht}
Let $H$ be either a string or a cycle with $\mu$ vertices. Then ${\rm grade}(I(\H))=\left\lceil\frac{\mu}{2}\right \rceil$.
\end{Remark}

\begin{Theorem}\label{CM}
Let $\H$ be a hypergraph with $\mu\geq 1$ vertices.
\begin{itemize}
\item[(i)] If $\H$ is a string, then $\H$ is Cohen-Macaulay if and only if $\mu=1$, or $\mu=3$ and $\H$ is not saturated.
\item[(ii)] If $H$ is a cycle, then $I(\H)$ is Cohen-Macaulay if and only if $\mu=3$ and $\H$ is not saturated, or $\mu=5$ and $\H$ does not contain two adjacent closed vertices.
\end{itemize}
\end{Theorem}
\begin{proof}
(i) If $\H$ is saturated, then $\pd(\H)=\mu$ (by Proposition \ref{saturated}) and ${\rm grade}\,I(\H)=\left\lceil\frac{\mu}{2}\right \rceil$ (by Remark \ref{ht}). These two numbers are equal if and only if $\mu=1$. Hence, in this case, $\H$ is Cohen-Macaulay if and only if $\mu(\H)=1$. We may then assume $\H$ is not saturated, and note that this implies $\mu\geq 3$.  

First, assume $\mu=3$. Since $\H$ is not saturated, then $\H$ contains exactly one open vertex, then, by Lemma \ref{3or4}, we have $\pd(\H)=3-1=2$ and, by Remark \ref{ht}, ${\rm grade}\,I(\H)=\left\lceil\frac{3}{2}\right \rceil=2$, proving the Cohen-Macaulay property.

For the converse, we show that if $\mu\geq 4$, then $\H$ is not Cohen-Macaulay. Let $\H_0$ be the string of opens with $\mu\geq 4$ vertices. 
By Remark \ref{ht}, we have ${\rm grade}\,I(\H)={\rm grade}\,I(\H_0)=\left\lceil\frac{\mu}{2}\right \rceil$ and, by Corollary \ref{examples}, $\pd(\H)\geq \pd(\H_0)=\mu-1-\left\lfloor\frac{\mu-3}{3}\right\rfloor=\mu-\left\lfloor\frac{\mu}{3}\right\rfloor$, hence it suffices to observe that $\mu-\left\lfloor\frac{\mu}{3}\right\rfloor>\left\lceil\frac{\mu}{2}\right \rceil$ for every $\mu\geq 4$. This follows by elementary combinatorics.

(ii) Let $\H_0$ be the $\mu$-cycle whose vertices are all open. By Corollary \ref{compare} we have $\pd(\H)\geq \pd(\H_0)=\left\lceil\frac{\mu}{3}\right\rceil+\left\lfloor\frac{\mu}{3}\right\rfloor$ and, by Remark \ref{ht}, we have ${\rm grade}\,I(\H)={\rm grade}\,I(\H_0)=\left\lceil\frac{\mu}{2}\right \rceil$.
 If $\mu\geq 6$,  by elementary combinatorics, we have $\left\lceil\frac{\mu}{3}\right\rceil+\left\lfloor\frac{\mu}{3}\right\rfloor> \left\lceil\frac{\mu}{2}\right \rceil$, hence
 $\pd(\H)\geq\pd(\H_0)> \left\lceil\frac{\mu}{2}\right \rceil$ and then $I(\H)$ is not Cohen-Macaulay. Similarly, if $\mu= 4$, then we have ${\rm grade}\,I(\H)={\rm grade}\,I(\H_0)=2$ and, by Lemma \ref{3or4}, $\pd(\H_0)=3$, proving that $\H$ is not Cohen-Macaulay. 
 
Hence, we only need to examine the cases $\mu=3$ or $\mu=5$. If $\mu=3$, then we have ${\rm grade}\,I(\H)={\rm grade}\,I(\H_0)=2$ and, by Lemma \ref{3or4}, $\pd(\H)=2$ unless $\H$ is saturated. Hence $I(\H)$ is Cohen-Macaulay if and only if $\H$ is not saturated.
On the other hand, if $\mu=5$, we have ${\rm grade}\,I(\H)={\rm grade}\,I(\H_0)=3$. 
If $\H$ has at most one closed vertex, then, by Theorem \ref{cycles2}, we have $\pd(\H)=3={\rm grade}\,I(\H)$. If $\H$ contains exactly two closed vertices and they are not adjacent, then, 
 by Theorem \ref{cycles2}, we have again $\pd(\H)=3={\rm grade}\,I(\H)$. Conversely, if $\H$ contains two adjacent closed vertices, then $\left\lfloor\frac{n_1(\H)-1}{3}\right\rfloor=0$ and one has either $s(\H)\leq 1$ and $M(\H)=0$, or $s(\H)=2$ and $M(\H)=1$. In either case, by Theorem \ref{cycles2}, we have $\pd(\H)\geq 5-1=4>3={\rm grade}\,I(\H)$, hence $I(\H)$ is not Cohen-Macaulay.
\end{proof}
Then, the only Cohen-Macaulay ideals have deviation at most two (in fact, they also appear in the classification \cite[4.9]{KTY2}).

\section{Appendix: algorithmic procedures and more examples}

The closed formula for the projective dimension of strings and cycles involves the modularity. Since it is time-consuming for a computer program to compute this invariant, we write here a simple algorithmic procedure to compute the projective dimension of any ideal associated to a string or cycle. The procedure for strings was anticipated in Remark \ref{proc}.
Since one may want to run the algorithm on a computer, we remark here that, by Theorems \ref{string} and \ref{cycles2}, the number $\pd(\H)$ is independent of the characteristic of the base field $k$ for all strings or cycles $\H$.

\begin{Algorithm}\label{alg}
Let $\H$ be a string hypergraph. 
\begin{itemize}
\item[(0)] Let $p_0=0$ and set a flag $i=1$.
\item[(1)] Check if $\H=\emptyset$. If so, then stop the procedure, and outputs $P=p_0+p_1+\ldots+p_i$.
\item[(2)] If $\H\neq\emptyset$, 
\begin{itemize}
\item if either $\mu(\H)=1$ or $\{2\}\in \H$, 
set $p_i=1$, $i=i+1$, $\H=\H_1$, and return to step (1);
\item if $\mu(\H)\geq 2$ and $\{2\}\notin \H$, 
set $p_i=2$, $i=i+1$, $\H=\H_3$, and return to step (1).
\end{itemize}
\end{itemize}
One has $\pd(\H)=P$.
\end{Algorithm}

\begin{Example}
Let $\H$ be the hypergraph of Example \ref{ex1}. Then Algorithm \ref{alg} gives $p_1=1$, $p_2=2$ and $p_3=2$, whence $\pd(\H)=1+2+2=5$.
\end{Example}

\begin{Example}
Let $\H$, $\H'$ and $\H''$ be the hypergraphs shown in Figure \ref{2S}. Then Algorithm \ref{alg} gives
$$\pd(\H)=2+1+1=4,\;\pd(\H')=2+2+2+2+1+1=10 \; \mbox{ and } \; \pd(\H'')=2+2+2+2+1+1=10.$$
\end{Example}

\begin{Example}
Let $\H'$, $\H''$  be the hypergraphs shown in Figure \ref{exchange}. Then Algorithm \ref{alg} gives
$$\pd(\H')=2+2+2=6 \quad \mbox{ and } \quad \pd(\H'')=2+1+2+1+1=7.$$
\end{Example}

\begin{Example}
Let $\H$, $\H'$  be the hypergraphs shown in Figure \ref{Switch}. Then Algorithm \ref{alg} gives
$$\pd(\H)=2+2+1+1+2+2+1=11 \quad \mbox{ and } \quad \pd(\H')=2+2+2+2+2+1=11.$$
\end{Example}

Using Lemma \ref{cons} and Lemma \ref{cycles1}, we can employ Algorithm \ref{alg} also to compute the projective dimension of cycles. Since the projective dimension of cycles with $\mu(\H)\leq 4$ is immediately computed by Lemma \ref{3or4}, we may assume $\mu(\H)\geq 5$.
\begin{Algorithm}\label{alg2}
Let $\H$ be a cycle hypergraph with $\mu(\H)\geq 5$.
\begin{itemize}
\item If $\H$ contains only open vertices, then $\pd(\H)=\mu-1-\lfloor\frac{\mu-2}{3}\rfloor$.
\item If $\H$ contains two consecutive closed vertices $v_1$ and $v_2$ , then $\pd(\H)=\pd(\s)$, where $\s$ is the string obtained by removing the face $\{v_1,v_2\}$;
\item If $\H$ contains a closed vertex, say $v_1$, whose neighbors are open, then $\pd(\H)={\rm max}\{\pd(\s_1),$ $ \pd(\s_5)+3\}$, where $\s_1$ and $\s_5$ are as in Lemma \ref{cycles1}.
\end{itemize}
Algorithm \ref{alg} -- applied to the strings $\s$, $\s_1$ and $\s_5$ -- now computes $\pd(\H)$.
\end{Algorithm}

%

\begin{Example}\label{c2}
Let $\H=\{ \{1\}, \{1,2\}, \{2,3\}, \{3,4\}, \{4,5\}, \{5,6\}, \{6\}, \{6,1\}  \}$. Then $\pd(\H)=\pd(\s)$, where $\s=\{ \{1\}, \{1,2\}, \{2,3\}, \{3,4\}, \{4,5\}, \{5,6\}, \{6\} \}$. By Algorithm \ref{alg}, we have $\pd(\s)=2+2$, whence $\pd(\H)=4$.
\end{Example}

\begin{Example}\label{c3}
Let $H$ be the hypergraph of Example \ref{max}, that is, $\H=\{ \{1\}, \{1,2\}, \{2,3\}, \{3\}, \{3,4\},$ $ \{4,5\}, \{5,6\}, \{6,7\}, \{7\}, \{7,8\}, \{8,1\}\}$, then one has
$\pd(\H)={\rm max}\{\pd(\s_1),\pd(\s_5)+3\}$. By Algorithm \ref{alg}, one has $\pd(\s_1)=1+2+1+1+1=6$ and $\pd(\s_5)=1+1=2$, showing that $${\rm max}\{\pd(\s_1),\pd(\s_5)+3\}={\rm max}\{6,5\}=6.$$
\end{Example}

\begin{Remark}\label{improve}
 Algorithm \ref{alg2} can be simplified by means of Proposition \ref{reduc}: if $\H$ contains strings of opens having $3$ or more open vertices, replace $\H$ by $\H'$ where $\H'$ is obtained by 'removing'' three connected open vertices from $\H$, and note that $\pd(\H)=\pd(\H')+2$. 
\end{Remark}

We illustrate Remark \ref{improve}. Let $\H$ be as in Example \ref{c2}, then by Proposition \ref{reduc} we have $\pd(\H)=\pd(\H')+2$, where $\H'$ is isomorphic to $\H''=\{\{1\}, \{1,2\}, \{2,3\}, \{3\}, \{3,1\}\}$. By Lemma \ref{3or4} we have $\pd(\H'')=2$, proving that $\pd(\H)=2+2=4$. We can now revisit Example \ref{c3} avoiding the computation of the projective dimension of the two strings $\s_1$ and $\s_5$.

\begin{Example}\label{c3rev}
Let $H$ be the hypergraph of Examples \ref{max} and \ref{c3}. By Proposition \ref{reduc}, one has
$\pd(\H)=\pd(\H')+2$, where $\H'$ is isomorphic to $\H'=\{\{1\}, \{1,2\}, \{2,3\}, \{3\}, \{3,4\}, \{4,5\}, \{5\},$ $ \{5,1\}\}$. By Algorithm \ref{alg2} we have $\pd(\H')=\pd(\s)$, where $\s=\{\{1\}, \{1,2\}, \{2,3\}, \{3\}, \{3,4\},$ 
$\{4,5\}, \{5\} \}$. By Algorithm \ref{alg} we have $\pd(\s)=2+2=4$, whence $\pd(\H)=2+4=6$.
\end{Example}

\end{document}